\documentclass{amsart}

\usepackage{amssymb}
\usepackage{amsfonts}
\usepackage{amsmath}
\usepackage{amsthm}
\usepackage{graphicx}
\usepackage{mathrsfs}
\usepackage{dsfont}
\usepackage{amscd}
\usepackage{multirow}
\usepackage{palatino}
\usepackage{mathpazo}
\usepackage{aurical}
\usepackage[all]{xy}
\usepackage[scaled]{helvet} 
\usepackage{courier} 
\normalfont
\usepackage[T1]{fontenc}
\usepackage{calligra}
\usepackage{verbatim}
\usepackage[usenames,dvipsnames]{color}
\usepackage[colorlinks=true,linkcolor=Blue,citecolor=Violet]{hyperref}
\usepackage{amsrefs}

\makeatletter
\@namedef{subjclassname@2010}{%
  \textup{2010} Mathematics Subject Classification}
\makeatother

\theoremstyle{plain}
\newtheorem{theorem}{Theorem}[section]

\newtheorem{lemma}[theorem]{Lemma}
\newtheorem{proposition}[theorem]{Proposition}

\newtheorem{theorem-definition}[theorem]{Theorem-Definition}

\theoremstyle{definition}
\newtheorem{definition}[theorem]{Definition}

\newtheorem{notation}[theorem]{Notation}

\newtheorem{convention}[theorem]{Convention}
\newtheorem{hypothesis}[theorem]{Hypthesis}

\theoremstyle{remark}

\newtheorem{example}[theorem]{Example}

\numberwithin{equation}{section}

\newcommand{\N}{{\mathds{N}}}

\newcommand{\R}{{\mathds{R}}}
\newcommand{\C}{{\mathds{C}}}

\newcommand{\D}{{\mathfrak{D}}}
\newcommand{\A}{{\mathfrak{A}}}
\newcommand{\B}{{\mathfrak{B}}}
\newcommand{\M}{{\mathfrak{M}}}
\newcommand{\F}{{\mathfrak{F}}}
\newcommand{\G}{{\mathfrak{G}}}

\newcommand{\Lip}{{\mathsf{L}}}

\newcommand{\qpropinquity}[1]{{\mathsf{\Lambda}_{#1}}}
\newcommand{\dpropinquity}[1]{{\mathsf{\Lambda}^\ast_{#1}}}

\newcommand{\Kantorovich}[1]{{\mathsf{mk}_{#1}}}

\newcommand{\StateSpace}{{\mathscr{S}}}

\newcommand{\mongekant}{{Mon\-ge-Kan\-to\-ro\-vich metric}}

\newcommand{\unit}{1}

\newcommand{\sa}[1]{{\mathfrak{sa}\left({#1}\right)}}

\newcommand{\dom}[1]{{\operatorname*{dom}\left({#1}\right)}}
\newcommand{\diam}[2]{{\mathrm{diam}\left({#1},{#2}\right)}}

\newcommand{\alg}[1]{{\mathfrak{#1}}}

\renewcommand{\geq}{\geqslant}
\renewcommand{\leq}{\leqslant}

\newcommand{\Latremoliere}{Latr\'{e}moli\`{e}re}
\allowdisplaybreaks[4]

\makeatletter
\newcommand{\vast}{\bBigg@{4}}
\newcommand{\Vast}{\bBigg@{5}}
\makeatother

\begin{document}

\title[Inductive limits and quantum metrics]{Inductive limits of C*-algebras and compact quantum metrics spaces}
\author{Konrad Aguilar}
\address{School of Mathematical and Statistical Sciences \\ Arizona State University \\  901 S. Palm Walk, Tempe, AZ 85287-1804}
\email{konrad.aguilar@asu.edu}
\urladdr{}

\date{\today}
\subjclass[2010]{Primary:  46L89, 46L30, 58B34.}
\keywords{Noncommutative metric geometry, Monge-Kantorovich distance, Quantum Metric Spaces, Lip-norms, inductive limits, AF algebras}

\begin{abstract}
Given a unital inductive limit of C*-algebras for which each C*-algebra of the inductive sequence comes equipped with a compact quantum metric of Rieffel, we produce sufficient conditions to build a compact quantum metric on the inductive limit from the quantum metrics on the inductive sequence by utilizing the completeness of the dual Gromov-Hausdorff propinquity of \Latremoliere\ on compact quantum metric spaces. This allows us to place new quantum metrics on all unital AF algebras that extend our previous work with \Latremoliere \ on unital AF algebras with faithful tracial state.  As a consequence, we produce a continuous image of the entire Fell topology on the ideal space of any unital AF algebra in the dual Gromov-Hausdorff propinquity topology.
\end{abstract}
\maketitle

\setcounter{tocdepth}{1}
\tableofcontents

\section{Introduction and Background}\label{s:int}
Noncommutative Metric Geometry introduced by Rieffel \cite{Rieffel98a, Rieffel00} and motivated by work of Connes \cite{Connes89, Connes} provides a solid framework for producing continuous families of quantum metric spaces, which are, in part, C*-algebras. This was done by producing a distance on the class of quantum metric spaces \cite{Rieffel00}, which is analogous to the Gromov-Hausdorff distance between compact metric spaces.  The first example of a continuous family in this metric was quantum tori with respect to their standard parametric space due to Rieffel \cite{Rieffel00}. Recently, \Latremoliere \ developed a distance on certain classes of quantum metric spaces,  called the Gromov-Hausdorff propinquity \cite{Latremoliere13, Latremoliere13b}, which allows one to capture the C*-algebraic structure as well as the quantum metric structure while also establishing continuity results about Hilbert C*-modules, group actions, and vector bundles associated to quantum metric spaces \cite{Latremoliere16, Latremoliere17, Latremoliere18, Latremoliere18b,  Rieffel15,  Rieffel18}.

This article focuses on the categorical notion of convergence of C*-algebras given by  inductive sequences of C*-algebras and their inductive limits, which  is a crucial topic of research in C*-algebras due, in part, to the Elliott classification program that began in \cite{Elliott76} where Elliott addressed the classification of inductive limits whose inductive sequences comprised of finite-dimensional C*-algebras. Due to the inception of Noncommutative Metric Geometry, the question of when this  categorical notion of convergence passes to a metric convergence arose naturally. But, first, to discuss convergence of C*-algebras, one must first equip the C*-algebras with the quantum metrics. There are many C*-algebras that have been equipped with quantum metrics including: certain Group C*-algebras \cite{Rieffel 02, Ozawa05, Rieffel15b}, quantum tori and fuzzy tori \cite{Rieffel98a, Latremoliere13c}, noncommutative solenoids \cite{Latremoliere-Packer17}, quantum Podle\'s sphere with D\polhk{a}browski-Sitarz spectral triple \cite{Aguilar-Kaad18, Dabrowski03}, and many more including untial AF algebras \cite{Aguilar16b, Antonescu04}. Now, in the case of unital AF algebras equipped with faithful tracial states, in \cite{Aguilar-Latremoliere15}, quantum metrics were placed in such a way that any inductive sequence that formed the given unital AF algebra converges to the AF algebra in the metric sense of propinquity so that the authors were able to also establish convergence of classes of AF algebras such as UHF algebras \cite{Glimm60} and Effros-Shen  algebras \cite{Effros80b}.  However, the types of quantum metrics built in the faithful tracial state case to acheive such convergence results  have not been extended to all unital AF algebras until this article (see Theorem \ref{t:main-af} and Proposition \ref{p:af-lip-compare}) even though there are quantum metrics on all unital AF algebras as mentioned above. And, we apply these findings to capture the entire Fell topology on the ideal space on any given unital AF algebra in the propinquity topology by way of a continuous map in Theorem \ref{t:ideal-conv}, which was not possible before since unital AF algebras may have ideals without faithful tracial states. This shows that the propinquity topology is diverse enough to capture this vital topology.

 However, our approach to produce these new quantum metrics on unital AF algebras is inspired by answering a more general question about quantum metrics on inductive limits that need not necessarily be AF. In every case of quantum metrics on inductive limits, the quantum metric is built directly on the inductive limit, which indirectly builds quantum metrics on the terms of the inductive sequence.  In our experience and as seen in this article, it  seems that in order to push results about quantum metrics on inductive limits to new avenues, one must have the ability to build quantum metrics on the inductive limit from suitable quantum metrics on the C*-algebras of the inductive sequence just as the C*-norm on an inductive limit of C*-algebras is constructed from the the C*-norms on the C*-algebras of the inductive sequence.  Thus, in this paper, we accomplish this task under suitable sufficient conditions to obtain our new quantum metrics on AF algebras. The main fact we use is that the Dual Gromov-Hausdorff Propinqiuity \cite{Latremoliere13b} is complete on certain classes of compact quantum metric spaces, and thus bestows a method for forming limits of quantum metric spaces.  Of course, this limit may not be an inductive limit of quantum metric spaces in any categorical sense, but the idea is to combine the categorical notion of inductive limit of C*-algebras with the metric limit formed by completeness with respect to propinquity. This is done in Theorem \ref{t:main} by using Cauchy sequences from inductive sequences of C*-algebras in dual propinquity that exploit the C*-inductive limit structure. We now introduce some background before we move onto our main results.
\begin{definition}[\cite{Rieffel98a,Rieffel99, Rieffel05}]\label{d:Monge-Kantorovich}
A {\em compact quantum  metric space} $(\A,\Lip)$ is an ordered pair where $\A$ is a unital C*-algebra with unit $1_\A$  and $\Lip$ is a seminorm over $\R$ defined on $\sa{\A}$  whose domain $\dom{\Lip}= \{ a \in \sa{\A} : \Lip (a) < \infty\}$ is a unital dense subspace of $\sa{\A}$ over $\R$ such that:
\begin{enumerate}
\item $\{ a \in \sa{\A} : \Lip(a) = 0 \} = \R\unit_\A$,
\item the \emph{\mongekant} defined, for all two states $\varphi, \psi \in \StateSpace(\A)$, by:
\begin{equation*}
\Kantorovich{\Lip} (\varphi, \psi) = \sup\left\{ |\varphi(a) - \psi(a)| : a\in\dom{\Lip}, \Lip(a) \leq 1 \right\}
\end{equation*}
metrizes the weak* topology of $\StateSpace(\A)$, and
\item the seminorm $\Lip$ is lower semi-continuous with respect to $\|\cdot\|_\A$.
\end{enumerate}
If $(\A,\Lip)$ is a  compact quantum metric space, then we call the seminorm $\Lip$ a {\em Lip-norm} and we denote the diameter of the compact metric space $(\StateSpace(\A), \Kantorovich{\Lip})$ by $\diam{\A}{\Lip}$.
\end{definition}

\begin{definition}[{\cite{Latremoliere15}}]\label{d:quasi-Monge-Kantorovich}
A {\em $(C,D)$-quasi-Leibniz compact quantum metric space} $(\A,\Lip)$, for some $C \in \R, C\geq 1$ and $D \in \R,D\geq 0$, is  compact quantum metric space such that
the seminorm $\Lip$ is  \emph{$(C,D)$-quasi-Leibniz}, i.e. for all $a,b \in \dom{\Lip}$:
\begin{equation*}
\max\left\{ \Lip\left(a \circ b \right), \Lip\left(\{a,b\}\right) \right\} \leq C\left(\|a\|_\A \Lip(b) + \|b\|_\A\Lip(a)\right) + D \Lip(a)\Lip(b)\text{,}
\end{equation*}
where $a\circ b = \frac{ab+ba}{2}$ is the Jordan product and $\{a,b\}=\frac{ab-ba}{2i}$ is the Lie product.

When $C=1, D=0$, we call $\Lip$ a {\em Leibniz} Lip-norm.  When we do not specify $C$ and $D$, we call $(\A,\Lip)$ a {\em quasi-Leibniz  compact quantum metric space}.
\end{definition}
Next, we define the  notion of isomorphism for compact quantum metric spaces, formerly called an isometric isomorphism.
\begin{definition}[{\cite[Definition 2.20]{Latremoliere13b}}]\label{d:q-iso}
Let $(\A, \Lip_\A), (\B, \Lip_\B)$ be two compact quantum metric spaces.  A *-isomorphism $\pi: \A \rightarrow \B$ is a {\em quantum isometry} if $\Lip_\B \circ \pi = \Lip_\A$.
\end{definition}

\begin{comment}
This definition is correct since for a *-isomorphism $\pi: \A \rightarrow \B$ to be a quantum isometry, it is necessary and sufficient for its dual map to be an isometry between the state spaces with respect to the associated quantum metrics (see \cite[6.2 Theorem]{Rieffel00}).
\end{comment}

Before we state the theorem that \Latremoliere's dual propinquity is a complete metric up to quantum isometry, we first define a crucial object used to provide estimates for the dual propinquity.

\begin{definition}[{\cite[Definition 2.25]{Latremoliere15}}]\label{d:tunnel}
Fix $C, \geq 1, D\geq 0$.  Let $(\A_1, \Lip_1), (\A_2, \Lip_2)$ be $(C,D)$-quasi-Leibniz compact quantum metric spaces. A {\em $(C,D)$-tunnel} $\tau$ from $(\A_1, \Lip_1)$ to  $(\A_2, \Lip_2)$ is a 4-tuple $\tau=(\D, \Lip_\D, \pi_1, \pi_2)$, where $(\D, \Lip_\D)$ is a $(C,D)$-quasi-Leibniz compact quantum metric space and for all $j \in \{1, 2\}$, the map $\pi_j: \D \rightarrow \A_j$ is a unital *-epimorphism such that for all $a \in \sa{\A_j}:$
\begin{equation*}
\Lip_j (a)=\inf \{ \Lip_\D (d) : d\in \sa{\D} \text{ and } \pi_j(d)=a\}.
\end{equation*}
\end{definition}
We take the following theorem from \cite{Latremoliere15} since it adapts the dual propinquity from \cite{Latremoliere13b} to the quasi-Leibniz case.
\begin{theorem}[{\cite[Definition 2.23, Theorem 2.28]{Latremoliere15}}]\label{t:dual-p}
Fix $C \geq 1, D \geq 0$.  There exists a pseudo-metric on the class of $(C,D)$-quasi-Leibniz compact quantum metric spaces called the dual Gromov-Hausdorff propinquity, denoted by $\dpropinquity{}$, such that for any two $(C,D)$-quasi-Leibniz compact quantum metric spaces $(\A, \Lip_\A), (\B, \Lip_\B)$
\begin{enumerate}
\item $\dpropinquity{}((\A, \Lip_\A), (\B, \Lip_\B))=0$ if and only if $(\A, \Lip_\A)$ and $(\B, \Lip_\B)$ are quantum isometric, and thus $\dpropinquity{}$ is a metric up to the equivalence relation of quantum isometry,
\item $\dpropinquity{}((\A, \Lip_\A), (\B, \Lip_\B)) \leq 2 \lambda(\tau)$, 

where $\tau$ is any $(C,D)$-tunnel from $(\A, \Lip_\A)$ to $(\B, \Lip_\B)$ and $\lambda(\tau)$ is the length of the tunnel (see \cite[Definition 2.25]{Latremoliere15}),
\end{enumerate}
and furthermore $\dpropinquity{}$ is a complete metric on the class of $(C,D)$-quasi-Leibniz compact quantum metric spaces up to quantum isometry.
\end{theorem}
We note futher that for any fixed $(C,D)$,  the dual propinquity is a noncommutative analogue of the Gromov-Hausdorff distance on the class of compact metric spaces up to isometry since one can embed this class homeomorphically into the class of $(C,D)$-quasi-Leibniz compact quantum metric up to quantum isometry with respect to the topology induced by $\dpropinquity{}$ (see \cite[Theorem 2.28]{Latremoliere15}).

Although we do not provide the definition of the length of a tunnel in this background, we provide a useful tool for bestowing estimates of the lengths of certain tunnels that arise in this paper. This tool was key in defining the quantum Gromov-Hausdorff propinquity, which is another metric on the class of $(C,D)$-quasi-Leibniz compact quantum metric up to quantum isometry that is likely not complete, but still has many applications, which includes finding estimates for the dual propinquity.
\begin{definition}[{\cite[Definition 3.6, Lemma 3.4]{Latremoliere13}}]\label{d:bridge}
Let $\A$, $\B$ be two unital C*-algebras.  A {\em bridge} $\gamma$ from $\A$ to $\B$ is a 4-tuple $\gamma=(\D, \omega, \pi_\A, \pi_\B)$ such that
\begin{enumerate}
\item $\D$ is a unital C*-algebra and $\omega \in \D$,
\item the set $\StateSpace_1(\omega)= \{ \psi \in \StateSpace(\D) : \forall d \in \D, \psi(d)=\psi(\omega d)=\psi(d \omega)\}$ is non-empty, in which case $\omega$ is called the {\em pivot}, and 
\item $\pi_\A : \A \rightarrow \D$ and $\pi_\B: \B \rightarrow \D$ are unital *-monomorphisms. 
\end{enumerate}
\end{definition}

\begin{comment}
A bridge always exists between two unital C*-algebras, $\A,\B$. Consider any C*-algebra formed the algebraic tensor product $\A \odot \B$, denoted by $\D$.  Then take $1_\D=1_\A \odot 1_\B$ as the pivot (here $\StateSpace_1(\omega)=\StateSpace(\D)$), and $ a\in \A \mapsto a \odot 1_\B$ and $b \in \B \mapsto 1_\A \odot b$ as the unital *-monomorphisms. 
\end{comment}

The next lemma produces a characterization of lengths of the types of bridges that appear in this article, which follows immediately from definition. But, first, we introduce a definition for the types of bridges that appear  in this article. 
\begin{definition}\label{d:e-bridge}
Let $\A$ be a unital C*-algebra, and let $\B \subseteq \A$ be a unital C*-subalgebra of $\A$.  We call the 4-tuple $(\A, 1_\A, \iota, \mathrm{id}_\A)$ the {\em evident bridge from $\B$ to $\A$}, where $\iota: \B \rightarrow \A$ is the inclusion mapping and $\mathrm{id}_\A: \A \rightarrow \A$ is the identity map. 
\end{definition}
\begin{lemma}[{\cite[Definition 3.17]{Latremoliere13}}]\label{l:bridge}
Let $(\A, \Lip_\A), (\B, \Lip_\B)$ be two  compact quantum metric spaces.  If a bridge $\gamma$ from $\A$ to $\B$ is of the form $\gamma=(\D, 1_\D, \pi_\A, \pi_\B)$, then the length of the bridge is
\begin{equation*}
\begin{split}
& \lambda(\gamma |\Lip_\A, \Lip_\B)=\\
&\max \left\{\begin{array}{c}
\sup_{a \in \sa{\A}, \Lip_\A(a) \leq 1} \left\{ \inf_{b \in \sa{\B}, \Lip_\B(b)\leq 1} \left\{ \|\pi_\A(a)-\pi_\B(b)\|_\D\right\}\right\}, \\
\ \ \ \ \sup_{b \in \sa{\B}, \Lip_\B(b) \leq 1} \left\{ \inf_{a \in \sa{\A}, \Lip_\A(a)\leq 1} \left\{ \|\pi_\A(a)-\pi_\B(b)\|_\D\right\}\right\}
\end{array}\right\}
\end{split}
\end{equation*}
In particular, this holds for evident bridges.
\end{lemma}
Next, we see how lengths of bridges can be used to estimate lengths of certain tunnels. We note that the length of any bridge between two compact quantum metric spaces is finite (see the discussion preceding \cite[Definition 3.14]{Latremoliere13}).
\begin{theorem}[{\cite[Theorem 3.48]{Latremoliere15b}}]\label{t:bridge-tunnel}
Fix $C\geq 1, D \geq 0$. Let $(\A, \Lip_\A), (\B, \Lip_\B)$ be two $(C, D)$-quasi-Leibniz compact quantum metric spaces.  Let $\gamma=(\D, \omega, \pi_\A, \pi_\B)$ be a bridge from $\A$ to $\B$. Fix any $r>\lambda(\gamma|\Lip_\A, \Lip_\B)$, where $\lambda(\gamma|\Lip_\A, \Lip_\B)$ is the length of the bridge $\gamma$. 

If we define for all $(a,b) \in \sa{\A \oplus \B}$
\[
\Lip^r_{\gamma|\Lip_\A, \Lip_\B}(a,b)=\max \left\{ \Lip_\A(a), \Lip_\B(b), \frac{\|\pi_\A(a)\omega - \omega \pi_\B(b)\|_\D}{r}\right\}
\]
and we let $p_\A: (a,b) \in \A \oplus \B \to a \in \A$ and $p_\B : (a,b) \in \A \oplus \B \to b \in \B$ denote the canonical surjections, then $\tau= (\A \oplus \B, \Lip^r_{\gamma|\Lip_\A, \Lip_\B}, p_\A, p_\B)$ is a $(C,D)$-tunnel from $(\A, \Lip_\A)$ to  $(\B, \Lip_\B)$ with length $\lambda(\tau) \leq r$, and 
\[
\dpropinquity{}\left(\left(\A, \Lip_\A\right), \left(\B, \Lip_\B\right)\right) \leq 2r.
\]
\end{theorem}
This allows us to define:
\begin{definition}\label{d:bridge-tunnel}
Fix $C\geq 1, D \geq 0$.  Let $(\A, \Lip_\A), (\B, \Lip_\B)$ be two $(C, D)$-quasi-Leibniz compact quantum metric spaces.  Let $\gamma=(\D, \omega, \pi_\A, \pi_\B)$ be a bridge from $\A$ to $\B$. We call the $(C,D)$-tunnel $(\A \oplus \B, \Lip^r_{\gamma|\Lip_\A, \Lip_\B}, p_\A, p_\B)$ from $(\A, \Lip_\A)$ to  $(\B, \Lip_\B)$ of Theorem \ref{t:bridge-tunnel} the {\em  $(r, \gamma|\Lip_\A, \Lip_\B)$-evident tunnel} associated to the bridge $\gamma$, Lip-norms $\Lip_\A, \Lip_\B$,  and $r >\lambda(\gamma|\Lip_\A, \Lip_\B).$
\end{definition}

\section{Quantum metrics on inductive limits by Cauchy sequences}

We now give a name to when a Cauchy sequence of quasi-Leibniz compact quantum metric spaces coming from an inductive sequence of C*-algebras does produce the inductive limit in the propinquity limit. In this section, we provide natural sufficient conditions for producing such a result in Theorem \ref{t:main} by harnessing the indcutive sequence structure. The following definition aims to describe when a Cauchy sequence of quasi-Leibniz compact quantum metric spaces is compatible with C*-algebraic structure of an inductive limt.
First, we provide a convention for inductive sequences and limits of C*-algebras. 
\begin{convention}\label{c:ind}
A unital C*-algebra $\A=\overline{\cup_{n \in \N} \A_n }^{\|\cdot \|_\A}$ is a  {\em unital  inductive limit of C*-algebras} if $(\A_n)_{n \in \N}$ is a non-decreasing sequence of unital C*-subalgebras of $\A$.
\end{convention}
The above convention captures all unital inductive limits of C*-algebras up to *-isomorphism, so there is no loss of generality by making such a convention \cite[Section 6.1]{Murphy90}.
\begin{definition}\label{d:c*-conv}
Fix $C \geq 1, D \geq 0$. Let $\A=\overline{\cup_{n \in \N} \A_n }^{\|\cdot \|_\A}$ be a unital   inductive limit of C*-algebras.  If $\left(\left(\A_n, \Lip_{\A_n}\right)\right)_{n \in \N}$ is a Cauchy sequence in dual propinquity of $(C,D)$-quasi-Leibniz compact quantum metric spaces with limit $(\B, \Lip_\B)$, then we call the sequence   $\left(\left(\A_n, \Lip_{\A_n}\right)\right)_{n \in \N}$ an {\em $\A-$C*-convergent sequence} if $\A$ is *-isomorphic to $\B$.
\end{definition}
From this definition, we see that the inductive limit $\A$ itself will be a $(C, D)$-quasi-Leibniz compact quantum metric space and a limit to the given Cauchy seqeunce in dual propinquity, and summarize this in:
\begin{proposition}\label{p:prop-ind-lim}
Fix $C \geq 1, D \geq 0$. Let $\A=\overline{\cup_{n \in \N} \A_n }^{\|\cdot \|_\A}$ be a unital   inductive limit of C*-algebras such that $\left(\left(\A_n, \Lip_{\A_n}\right)\right)_{n \in \N}$ is a Cauchy sequence of $(C,D)$-quasi-Leibniz compact quantum metric spaces  with dual propinquity limit $(\B, \Lip_\B)$.

 If $\left(\left(\A_n, \Lip_{\A_n}\right)\right)_{n \in \N}$ is a $\A-$C*-convergent for some *-isomorphism $\pi: \A \rightarrow \B$, then $\Lip^\B_\A := \Lip_\B \circ \pi$ is a $(C,D)$-quasi-Leibniz Lip-norm on $\A$    such that  
\[
\dpropinquity{} \left(\left(\B, \Lip_\B\right), \left( \A, \Lip^\B_\A \right)\right) =0 \text{ \ and \ } \lim_{n \to \infty} \dpropinquity{} \left(\left(\A_n, \Lip_{\A_n}\right), \left( \A, \Lip^\B_\A \right)\right) = 0.
\]
\end{proposition}
\begin{proof}
It is routine to check that $\Lip^\B_\A$ is a $(C,D)$-quasi-Leibniz Lip-norm on $\A$ such that  $\dpropinquity{} \left(\left(\B, \Lip_\B \right), \left( \A, \Lip^\B_\A \right)\right) =0$ since $\pi$ is a quantum isometry by construction.  The convergence result follows from the triangle inequality.
\end{proof}
A key observation is that: {\em  it can be the case that an inductive limit of an inductive sequence of C*-algebras equipped with quantum metrics, which happen to produce a Cauchy sequence in dual propinquity, need not be *-isomorphic to the C*-algebra of the limit in the dual propinquity.} Let us now provide an example of non-C*-convergent sequence built from the CAR algebra \cite[Example III.5.4]{Davidson} to motivate sufficient conditions that produce C*-convergent sequences. We note that the following example is motivated by work in \cite{Aguilar-Latremoliere17}.
\begin{example}[A non-C*-convergent sequence]\label{e:non-c*-cauchy-ex}
Consider the inductive limit, the CAR algebra $\M_{2^\infty}= \overline{\cup_{n \in \N}(\M_{2^\infty})_n}^{\|\cdot \|_{\M_{2^\infty}}}$ \cite[Example III.5.4]{Davidson}, where $(\M_{2^\infty})_n \cong M_{2^n}(\C)$ for all $n \in \N$.

For each $n \in \N$, define $\Lip_{(\M_{2^\infty})_n }(a) = \dim((\M_{2^\infty})_n)\cdot \|a-\tau(a)1_{\M_{2^\infty}}\|_{\M_{2^\infty}}$ for all $a \in (\M_{2^\infty})_n$, where $\tau$ is the unique faithful tracial state on $\M_{2^\infty}$. By \cite[Theorem 3.5]{Aguilar-Latremoliere15}, the pair $\left((\M_{2^\infty})_n, \Lip_{(\M_{2^\infty})_n }\right)$ is a $(2,0)$-quasi Leibniz compact quantum metric space.

Now, consider $\sigma_n : \lambda \in \C \mapsto \lambda 1_{(\M_{2^\infty})_n} \in (\M_{2^\infty})_n $ and $\sigma_n(\C)$ with its unique Lip-norm $\Lip_\C$, the 0-seminorm, which is also $(2,0)$-quasi-Leibniz. Clearly $\sigma_n(\C) \not\cong \M_{2^\infty}$.  Fix $n \in \N$, consider the evident bridge from $\sigma_n(\C)$ to $(\M_{2^\infty})_n$ (Definition \ref{d:e-bridge}). Next, let $\lambda \in \sigma_n(\C)$ (so $\Lip_\C(\lambda)= 0 \leq 1$), then $\| \lambda-\lambda\|_{\M_{2^\infty}}=0$ where $\Lip_{(\M_{2^\infty})_n }(\lambda)=0$.  Now, let $ a\in (\M_{2^\infty})_n $ such that $\Lip_{(\M_{2^\infty})_n }(a)  \leq 1$, then $\|\sigma_n(\tau(a))-a\|_{\M_{2^\infty}} \leq 1/\dim((\M_{2^\infty})_n)$ where $\Lip_\C(\sigma_n(\tau(a))=0 \leq 1$.  Thus, by Lemma \ref{l:bridge}, the length of the bridge is less than or equal to $ 1/\dim((\M_{2^\infty})_n)$. Hence, by  and Theorem \ref{t:bridge-tunnel}, we have that 
\[
\qpropinquity{}^*\left(\left(\C, \Lip_\C\right), \left((\M_{2^\infty})_n, \Lip_{(\M_{2^\infty})_n }\right)\right) \leq \frac{4}{\dim((\M_{2^\infty})_n)}
\]
and thus $\lim_{n \to \infty} \qpropinquity{}^*\left(\left(\C, \Lip_\C\right), \left((\M_{2^\infty})_n, \Lip_{(\M_{2^\infty})_n }\right)\right)=0$. Thus, the  Cauchy sequence  $\left( \left((\M_{2^\infty})_n, \Lip_{(\M_{2^\infty})_n }\right)\right) $ is not $\M_{2^\infty}-$C*-convergent.
\end{example}
The above example is interesting on its own as it reflects the fact that matrices can approximate any unital commutative C*-algebra in the Gromov-Hausdorff propinquity \cite{Aguilar-Latremoliere17}. However, this is not suitable for our current pursuits.  Now, back to our main goal, it should be noted that there do exist Lip-norms $\Lip_n'$ on $(\M_{2^\infty})_n$ such that the sequence $(((\M_{2^\infty})_n,\Lip_n'))_{n \in \N}$ is Cauchy in  dual propinquity and is $\M_{2^\infty}-$C*-con\-vergent by \cite[Theorem 3.5]{Aguilar-Latremoliere15}.
We would now like to provide some sufficient conditions to build C*-convergent  sequences.  The above example can motivate such conditions.  Indeed, consider the inclusion mapping $\iota_n : (\M_{2^\infty})_n \rightarrow (\M_{2^\infty})_{n+1}$ and let $a \in (\M_{2^\infty})_n$.  We have 
\begin{equation*}
\begin{split}\Lip_{(\M_{2^\infty})_{n+1}}(\iota_n(a))&= \dim((\M_{2^\infty})_{n+1})\cdot \|a-\tau(a)1_{\M_{2^\infty}}\|_{\M_{2^\infty}}\\
&= \frac{\dim((\M_{2^\infty})_{n+1})}{\dim((\M_{2^\infty})_{n})} \Lip_{(\M_{2^\infty})_{n}}(a).
\end{split}
\end{equation*} Therefore $\Lip_{(\M_{2^\infty})_{n+1}}(\iota_n(a))\not\leq  \Lip_{(\M_{2^\infty})_{n}}(a).$  Thus, the canonical embeddings are not contractive with respect to the Lip-norms, which is not very compatible with the notion of an inductive limit. For instance, the embeddings for an inductive sequence of C*-algebras are contractive with respect to the C*-norms. Thus, requiring contractivity seems to be a desirable condition along with the fact that the Lip-norms on AF algebras in \cite{Aguilar-Latremoliere15, Aguilar16b, Antonescu04} are all contractive. Furthermore, the Lip-norms in \cite{Aguilar-Latremoliere15} allowed for explicit upper bounds on the lengths of the evident bridges of Definition \ref{d:e-bridge} in dual propinquity. From this, the authors  showed convergence of AF algebras and not just convergence of the inductive sequences that formed the AF algebras. However, in  \cite{Aguilar-Latremoliere15}, these methods only worked for AF algebras with faithful tracial states. Thus, as a consequence of this paper and Section \ref{s:af}  and the following definition, we will have similar convergence results for all AF algebras with or without faithful tracial state by building Lip-norms on inductive limits from the Lip-norms on the terms of the inductive sequence.  Of course, the Lip-norms on the inductive limits of \cite{Aguilar-Latremoliere15} were built explicitly on the inductive limit, so our Lip-norms in this paper would be redundant in the faithful tracial state case and up to passing to a subsequence these inductive sequences satisfy the following definition automatically, and of course C*-convergent with respect to the given inductive limits. Thus, {\em  another main purpose of this paper is bestow a method to construct Lip-norms on inductive limits from Lip-norms on  certain inductive sequences without knowledge of any quantum metric structure on the inductive limit itself.}
\begin{definition}\label{d:approx-ind}
Fix $C \geq 1, D \geq 0$. Let $\A=\overline{\cup_{n \in \N} \A_n }^{\|\cdot \|_\A}$ be a unital   inductive limit of C*-algebras.  Let $\left(\left(\A_n, \Lip_{\A_n}\right)\right)_{n \in \N}$ be a sequence of $(C,D)$-quasi-Leibniz compact quantum metric spaces. We the call the inductive limit $\A$ an {\em  $\left(\left(\A_n, \Lip_{\A_n}\right)\right)_{n \in \N}$-propinquity approximable} inductive limit if the following hold for each $n \in \N$:
\begin{enumerate}
\item  there exists a $(C, D)$-quasi-Leibniz Lip-norm $\Lip_{\A_n}$ for $\A_n$ such that $\Lip_{\A_n}$ is defined on $\A_n$ and  $\{ a\in \A_n :\Lip_{\A_n}(a)< \infty\}$ is a dense *-subalgebra of $\A_n$,
\item  it holds that if $a \in \A_n$, then $\Lip_{\A_{n+1}}(a) \leq \Lip_{\A_n}(a)$, and
\item there exists a sequence $(\beta(j))_{j \in \N} \subset (0, \infty)$ such that $\sum_{j=0}^\infty \beta(j) < \infty$ and the length of evident bridge $\gamma_{n,n+1}=(\A_{n+1}, 1_\A, \iota_{n,n+1}, \mathrm{id}_{n+1})$ of Definition (\ref{d:e-bridge}),  satisfies
\[
\lambda(\gamma_{n,n+1}|\Lip_{\A_n}, \Lip_{\A_{n+1}}) \leq \beta(n),
\] and we denote the associated $(2\beta(n), \gamma_{n,n+1}|\Lip_{\A_n}, \Lip_{\A_{n+1}})-$evident tunnel of Definition \ref{d:bridge-tunnel} by $\tau_{n,n+1}$.
\end{enumerate}
\end{definition}
We note that the above definition satisifies the notion of an inductive sequence in a certain category of compact quantum metric spaces defined in \cite[Definition 1.9]{Latremoliere17}, but so does the sequence in Example \ref{e:non-c*-cauchy-ex}. Thus, the above definition is an attempt to provide further criteria to allow the inductive sequence to form some sort of limit even though it is not an inductive limit. In particular,
 the purpose of this definition is to focus on a situation that is capable of allowing for convergence of inductive limits themselves by providing explicit estimates from the inductive sequences as seen in Theorem \ref{t:sequence}. Approximability occured in and was motivated by the work in \cite{Aguilar-Latremoliere15}.  However, we recall that this only occured in the case for AF algebras equipped with faithful tracial states. Approximable quantum metrics have not yet been provided for AF algebras outside the faithful tracial state case even by the Lip-norms on all AF algebras built from quotient norms in \cite{Aguilar16b} and spectral triples in \cite{Antonescu04}. The fact that the inductive sequences converge to the inductive limit in these cases came from a compactness argument, which did not provide explicit estimates assocaited to evident bridges (see \cite[Theorem 4.10 and Remark 4.11]{Aguilar16b}).  Yet, there are still advantages to these quantum metrics as they preserve more algebraic structure than the ones of this paper and \cite{Aguilar-Latremoliere15} since they are {\em strongly Leibniz} of \cite[Definition 2.1]{Rieffel10} and have domains that preserve taking inverses; hence, these Lip-norms of  \cite{Aguilar16b,Antonescu04} still require more investigation, which lies outside the context of this paper.

We  now begin the journey to show that Definition \ref{d:approx-ind} gifts C*-convergent sequences. We begin with:
\begin{proposition}\label{p:c*-cauchy}
Fix $C \geq 1, D \geq 0$. Let $\A=\overline{\cup_{n \in \N} \A_n }^{\|\cdot \|_\A}$ be a unital   inductive limit of C*-algebras. 

If $\A$ is $\left(\left(\A_n, \Lip_{\A_n}\right)\right)_{n \in \N}$-propinquity approximable for some sequence of $(C,D)$-quasi-Leibniz compact quantum metric spaces and summable $(\beta(j))_{j \in \N} \subset (0, \infty)$, then:
\begin{enumerate}
 \item the sequence $\left(\left(\A_n, \Lip_{\A_n}\right)\right)_{n \in \N}$ is Cauchy in dual propinquity, where for  $n \in \N$
 \[\dpropinquity{}\left(\left(\A_n, \Lip_{\A_n}\right), \left( \A_{n+1}, \Lip_{\A_{n+1}}\right)\right) \leq 2 \lambda(\tau_{n,n+1}) \leq 4 \beta(n), \text{ and }
 \]
 \item given any $(C,D)$-quasi-Leibniz limit $(\B, \Lip_\B)$ up to quantum isometry  of the Cauchy sequence $\left(\left(\A_n, \Lip_{\A_n}\right)\right)_{n \in \N}$, it holds that 
 \[ \dpropinquity{}\left(\left(\A_n, \Lip_{\A_n}\right), \left(\B, \Lip_\B \right)\right) \leq 4 \sum_{j=n}^\infty \beta(j)\text{ \quad  for all $n \in \N.$}\]
\end{enumerate}
\end{proposition}
\begin{proof}
The inequalities of (1) follow immediately from Theorem \ref{t:bridge-tunnel}.  The fact that $\left(\left(\A_n, \Lip_{\A_n}\right)\right)_{n \in \N}$ is Cauchy follows from the fact that $(\beta(j))_{j \in \N}$ is summable. Conclusion (2) follows from the triangle inequality.
\end{proof}
Thus, if we find a limit  $(\B, \Lip_\B)$ of $\left(\left(\A_n, \Lip_{\A_n}\right)\right)_{n \in \N}$ such that $\A \cong \B$, then $\left(\left(\A_n, \Lip_{\A_n}\right)\right)_{n \in \N}$ will be $\A-$C*-convergent.  Thankfully, \Latremoliere \ provided a succinct construction of a limit of  certain Cauchy sequences in dual propinquity in \cite[Section 6]{Latremoliere13b}, which \Latremoliere \  used to prove that the dual propinquity is complete.  Hence, we now provide a summary of the construction.
\begin{notation}[{\cite[Hypothesis 6.2, Lemma 6.19, and Lemma 6.20]{Latremoliere13b}}]\label{n:cauchy-lim}
Fix $C \geq 1, D \geq 0$. Let $\left(\left(\A_n, \Lip_{\A_n}\right)\right)_{n \in \N}$ be sequence of $(C, D)$-quasi-Leibniz compact quantum metric spaces. For each $n \in \N$, let  $\tau_{\A_n}=(\D_n, \Lip^n, \pi_n, \omega_n)$ be a $(C,D)$-tunnel from $\left(\A_n, \Lip_{\A_n}\right)$ to $ \left(\A_{n+1}, \Lip_{\A_{n+1}}\right)$ and denote $\tau_{\A_\N}=(\tau_{\A_n})_{n \in \N}$.

We consider $\D:=\prod_{n \in \N} \D_n$ as the unital C*-algebra of bounded sequences.

Define the seminorm
\begin{equation*}
S_0 : (d_n)_{n \in \N} \in \sa{\prod_{n \in \N} \D_n} \longmapsto \sup \{ \Lip^n(d_n) : n \in \N\}
\end{equation*}
and 
\begin{equation*}
\begin{split}
& \alg{K}_0:=\left\{ d=(d_n)_{n \in \N} \in \sa{\D} \mid  (\forall n \in \N)  \pi_{n+1}(d_{n+1})=\omega_n(d_n) \right\} \\
&\alg{L}_0 := \left\{ d=(d_n)_{n \in \N} \in \alg{K}_0 \mid S_0(d) < \infty  \right\}
\end{split}
\end{equation*}
 We denote $\alg{Re}(d)=(d+d^*)/2$ and $\alg{Im}(d)=(d-d^*)/(2i)$ for all $ d\in \prod_{n \in \N} \D_n$.
Define
\begin{equation*}
\alg{S}_0:=\left\{d=(d_n)_{n \in \N} \in \D \mid \alg{Re}(d),\alg{Im}(d) \in  \alg{L}_0 \right \} \text{ \ and \ } \alg{G}_0:=\overline{\alg{S}_0}^{\|\cdot\|_\D}.
\end{equation*}

Next, define the subspace
\begin{equation*}
\alg{I}_0 := \left\{ (d_n)_{n \in \N} \in \alg{G}_0 \mid  \lim_{n \to \infty }\|d_n\|_{\D_n} = 0 \right\},
\end{equation*}
and
\begin{equation*}\F^{\tau_{\A_\N}}:=\alg{G}_0/\alg{I}_0.
\end{equation*}  

Now, let $q: \alg{G}_0 \rightarrow \F^{\tau_{\A_\N}}$ be the quotient map and define for all $a \in \sa{\F^{\tau_{\A_\N}}}$
\begin{equation*}
\Lip_{\tau_{\A_\N}}(a):=\inf \{ S_0(d): d \in \sa{\alg{G}_0} \text{ and } q(d)=a \}.
\end{equation*}
\end{notation}
With this, we state the key result of \cite{Latremoliere13b} that leads immediately to the proof of completeness of the dual propinquity. 
\begin{theorem}[{\cite[Section 6, Proposition 6.26]{Latremoliere13b}},  {\cite[Theorem 2.28]{Latremoliere15}}]\label{t:d-prop-lim}
 Fix $C \geq 1, D \geq 0$.  Let $\left(\left(\A_n, \Lip_{\A_n}\right)\right)_{n \in \N}$ be sequence of $(C, D)$-quasi-Leibniz compact quantum metric spaces. For each $n \in \N$, let  $\tau_{\A_n}=(\D_n, \Lip^n, \pi_n, \omega_n)$ be a $(C,D)$-tunnel from $\left(\A_n, \Lip_{\A_n}\right)$ to $ \left(\A_{n+1}, \Lip_{\A_{n+1}}\right)$.
 
 Using Notation \ref{n:cauchy-lim}, if  $\sum_{n=0}^\infty \lambda(\tau_{\A_n}) < \infty$, then:
\begin{enumerate}
\item $\alg{G}_0$ is a unital C*-subalgebra of $\D$, 
\item $\alg{I}_0$ is a closed two-sided ideal of $\alg{G}_0$, and thus $\F^{\tau_{\A_\N}}$ is a unital C*-algebra, 
\item $\left(\F^{\tau_{\A_\N}}, \Lip_{\tau_{\A_\N}}\right)$ is a $(C,D)$-quasi-Leibniz compact quantum metric space, 
\item the sequence $\left(\left(\A_n, \Lip_{\A_n}\right)\right)_{n \in \N}$ is Cauchy in dual propinquity, and 
\[
\lim_{n \to \infty} \dpropinquity{}\left( \left(\A_n, \Lip_{\A_n}\right),\left(\F^{\tau_{\A_\N}}, \Lip_{\tau_{\A_\N}}\right)\right)=0.
\]
\end{enumerate}
\end{theorem}
Next, we gather more detail about the objects of Notation \ref{n:cauchy-lim} in the setting of approximable inductive limits. But, first we state some hypotheses we will use in the next few theorems and definitions.
\begin{hypothesis}\label{h:ind-lim}
Fix $C \geq 1, D \geq 0$.  Let $\A=\overline{\cup_{n \in \N} \A_n }^{\|\cdot \|_\A}$ be a unital   inductive limit of C*-algebras.  Let $\A$ be $\left(\left(\A_n, \Lip_{\A_n}\right)\right)_{n \in \N}$-propinquity approximable for some sequence of $(C,D)$-quasi-Leibniz compact quantum metric spaces and summable $(\beta(j))_{j \in \N} \subset (0, \infty)$.

  Let $\tau_{\A_n}=\tau_{n,n+1}$ of Notation \ref{n:cauchy-lim} be the   $(2\beta(n), \gamma_{n,n+1} |\Lip_{\A_n}, \ \Lip_{\A_{n+1}})$-evident tunnel from $(\A_n, \Lip_{\A_n})$ to $(\A_{n+1}, \Lip_{\A_{n+1}})$ of Definition \ref{d:approx-ind} for each $n \in \N$.
\end{hypothesis}
\begin{proposition}\label{p:b-c*-conv-lim}
If we assume Hypothesis \ref{h:ind-lim}, then from Notation \ref{n:cauchy-lim}
\begin{enumerate}
\item $\alg{K}_0= \left\{ ((a_n^n, a^n_{n+1}))_{n \in \N} \in \sa{\prod_{n \in \N} \A_n \oplus \A_{n+1}}\mid (\forall n \in \N) a^n_{n+1}=a^{n+1}_{n+1}\right\} $, 
\item $1_\D=1_{\alg{G}_0}=((1_\A, 1_\A))_{n \in \N}$ and  for all $d=((a_n^n, a^n_{n+1}))_{n \in \N} \in \alg{K}_0$, we have
\[
S_0(d) = \sup_{n \in \N}\left\{ \max \left\{ \Lip_{\A_n}\left(a_n^n \right), \frac{\left\|a^n_n-a^{n+1}_{n+1}\right\|_\A}{2\beta(n)}\right\}\right\}, 
\]
\item and $\lim_{n \to \infty} \dpropinquity{}\left( \left(\A_n, \Lip_{\A_n}\right),\left(\F^{\tau_{\A_\N}}, \Lip_{\tau_{\A_\N}}\right)\right)=0$, where for $n \in \N$,
\[
 \dpropinquity{}\left(\left(\A_n, \Lip_{\A_n}\right),\left(\F^{\tau_{\A_\N}}, \Lip_{\tau_{\A_\N}}\right)\right) \leq 4 \sum_{j=n}^\infty \beta(j)
.\]
\end{enumerate} 
\end{proposition}
\begin{proof}
Apply Definition \ref{d:approx-ind} to Notation \ref{n:cauchy-lim} to obtain conclusions (1) and (2). Conclusion (3) follows immediately from Proposition \ref{p:c*-cauchy} and Theorem \ref{t:d-prop-lim}.
\end{proof}
The above proposition shows us that $\F^{\tau_{\A_\N}}$ is beginning to look  like an inductive limit of C*-algebras itself given Definition \ref{d:approx-ind}. Thus, next, we begin building our *-isomorphism from $\A$ onto $\F^{\tau_{\A_\N}}$.  To do this, we follow the standard method for providing *-isomorphisms from inductive limits by universality \cite[6.1.2 Theorem]{Murphy90}.

\begin{definition}\label{d:*-mon}
Assuming Hypothesis \ref{h:ind-lim}, 
define $\psi_0: \A_0 \rightarrow \D$ by
\begin{equation*}
\psi_0(a_0)=((a_0,a_0))_{n \in \N},
\end{equation*}
and 
for $n \in \N \setminus \{0\}$, define $\psi_n : \A_n \rightarrow \D$ by 
\begin{equation*}
\psi_n(a_n)=((0,0), \ldots, (0,0), (0,a_n),(a_n, a_n),(a_n,a_n), \ldots ),
\end{equation*}
where $(0,a_n) \in \D_{n-1}=\A_{n-1}\oplus \A_n$.
\end{definition}

\begin{lemma}\label{l:*-mon}
Assuming Hypothesis \ref{h:ind-lim}, the map $\psi_n : \A_n \rightarrow \D$ is a *-monomorphism such that $\psi_n(\A_n) \subseteq \G_0$ for all $n \in \N$.
\end{lemma}
\begin{proof}
Fix $n \in \N$. The fact that $\psi_n $ is a *-monomorphism is clear. Let $a \in \A_n$ such that $\Lip_{\A_n}(a) < \infty$.  If $a \in \C1_\A$, then $S_0(\psi_n(a))=0< \infty$. So, assume $a \not\in \C1_\A$.  Thus $\Lip_{\A_n}(\alg{Re}(a))< \infty,  \Lip_{\A_n}(\alg{Im}(a))< \infty$ by  (1) of Definition \ref{d:approx-ind}. In particular, since $\psi_n(\A_n) \subseteq \alg{K}_0$ by Proposition \ref{p:b-c*-conv-lim} , we have $S_0(\alg{Re}(\psi_n(a)))= \max \{\Lip_{\A_n} (\alg{Re}(a)), \|\alg{Re}(a)\|_\A/(2\beta(n-1))\} < \infty $ and similarly $S_0(\alg{Im}(\psi_n(a)))< \infty $ by Proposition \ref{p:b-c*-conv-lim} and (2) of Definition \ref{d:approx-ind}.  Therefore, by (1) of Definition \ref{d:approx-ind}
\begin{equation*}
\begin{split}
\psi_n(\A_n)& = \psi_n\left(\overline{\{a \in \A_n : \Lip_{\A_n}(a) < \infty\}}^{\| \cdot \|_\A}\right)\\
& \subseteq \overline{\psi_n\left( \{a \in \A_n : \Lip_{\A_n}(a) < \infty\}\right)}^{\| \cdot \|_\D} \\
& \subseteq \overline{\alg{S}_0}^{\| \cdot \|_\D}= \G_0.
\end{split}
\end{equation*}
by continuity.
\end{proof}
This lemma allows us to define:
\begin{definition}\label{d:u*-mon}
Assuming Hypothesis \ref{h:ind-lim}, for each $n \in \N$, by Lemma \ref{l:*-mon} we may define
 $\psi^{(n)} : \A_n \rightarrow \F^{\tau_{\A_\N}}$ by 
\begin{equation*}
\psi^{(n)} :=q \circ \psi_n
\end{equation*}
where $q : \G_0 \rightarrow \F^{\tau_{\A_\N}}$ is the quotient map.
\end{definition}

\begin{lemma}\label{l:u*-mon}
Assuming Hypothesis \ref{h:ind-lim}, the map
$\psi^{(n)}:\A_n \rightarrow \F^{\tau_{\A_\N}}$ is a unital *-mono\-morphism for each $n \in \N$. Furthermore $\psi^{(n)}=\psi^{(n+1)}$ on $\A_n$ for all $n \in \N$.
\end{lemma}
\begin{proof}
Fix $n \in \N$. The map  $\psi^{(n)}$  is clearly a  *-homomorphism.  For unital, we note that $\psi_n(1_\A)-1_\D \in \alg{I}_0$, and thus $\psi^{(n)}(1_\A)=1_{\F^{\tau_{\A_\N}}}.$  For injectivity, assume $a,b \in \A_n$ and $\psi^{(n)}(a)=\psi^{(n)}(b)$.  Thus $\psi_n(a)-\psi_n(b) \in \alg{I}_0$. Hence $
0 = \lim_{n \to \infty } \|a-b\|_\A = \|a-b\|_\A$ 
which implies $a=b$. Finally, let $a \in \A_n\subseteq \A_{n+1}$. Then, again we have  $\psi_n(a)-\psi_{n+1}(a) \in \alg{I}_0$, and thus $\psi^{(n)}(a)=\psi^{(n+1)}(a).$
\end{proof}
Next, we prove the main theorem of this section, where we see all the notions introduced thus far come together.

\begin{theorem}\label{t:main}
Fix $C \geq 1, D \geq 0$.  Let $\A=\overline{\cup_{n \in \N} \A_n }^{\|\cdot \|_\A}$ be a unital   inductive limit of C*-algebras. 

If $\A$ is $\left(\left(\A_n, \Lip_{\A_n}\right)\right)_{n \in \N}$-propinquity approximable for some sequence of $(C,D)$-quasi-Leibniz compact quantum metric spaces and summable $(\beta(j))_{j \in \N} \subset (0, \infty)$, then the sequence $\left(\left(\A_n, \Lip_{\A_n}\right)\right)_{n \in \N}$ is $\A-$C*-convergent with respect to  $\left(\F^{\tau_{\A_\N}}, \Lip_{\tau_{\A_\N}}\right)$, where:
\begin{enumerate}
\item $\tau_{\A_\N}=(\tau_{n,n+1})_{n \in \N}$ are the evident tunnels from (3) of Definition \ref{d:approx-ind},  
\item there exists a unital *-isomorphism $\psi: \A \rightarrow \F^{\tau_{\A_\N}}$ such that  $\psi=\psi^{(n)}$ on $\A_n$ for all $n \in \N$ of Definition \ref{d:u*-mon}, and
\item if we define $\Lip_\A:= \Lip_{\tau_{\A_\N}} \circ \psi$, then $(\A, \Lip_\A)$ is a $(C, D)$-quasi-Leibniz compact quantum metric space such that $\cup_{n \in \N} \dom{\Lip_{\A_n}} \subseteq \dom{\Lip_\A}$ with 
\[\qpropinquity{}^* \left(\left(\A_n, \Lip_n\right), \left(\A, \Lip_\A \right)\right) \leq 4 \sum_{j=n}^\infty \beta(j) \text{ \quad and \quad } \lim_{n \to \infty} \qpropinquity{}^* \left(\left(\A_n, \Lip_n\right), \left(\A, \Lip_\A \right)\right)=0.\]
\end{enumerate}
\end{theorem}
\begin{proof}
The fact that there exists a unital *-monomorphism $\psi: \A \rightarrow \F^{\tau_{\A_\N}}$ such that  $\psi=\psi^{(n)}$ on $\A_n$ for all $n \in \N$ follows from \cite[6.1.2 Theorem]{Murphy90} and Lemma \ref{l:u*-mon}. 

Next, we show $\psi(\A)=\F^{\tau_{\A_\N}}.$ Let $a+\alg{I}_0 \in \F^{\tau_{\A_\N}}$. Let $\varepsilon>0$.  There exists $b=(b_n)_{n \in \N}=((b_n^n, b^n_{n+1}))_{n \in \N} \in \alg{S}_0$ such that $\|a+\alg{I}_0 - b+\alg{I}_0\|_\F < \varepsilon/2$ by density.  Thus $\alg{Re}(b), \alg{Im}(b) \in \alg{L}_0$. Hence, there exists $r\in \R, r>0$ such that $S_0(\alg{Re}(b))\leq r, S_0(\alg{Im}(b))\leq r$.  There exists $N \in \N, N>1$ such that $2r \cdot \sum_{j=N}^\infty \beta(j) < \varepsilon/4$.  

Next, note that $\alg{Re}(b)=((\alg{Re}(b_n^n), \alg{Re}(b^n_{n+1})))_{n \in \N}$. Define $c=((c_n^n, c^n_{n+1}))_{n \in \N} \in \D$ in the following way:  
\[(c_n^n, c^n_{n+1})=\begin{cases}(0,0) & : 0 \leq n \leq N-2\\
(0,\alg{Re}(b^{N-1}_{N})) & : n=N-1 \\
(\alg{Re}(b_n^n), \alg{Re}(b^n_{n+1}))& : n \geq N.
\end{cases} 
\]Therefore $c \in \alg{L}_0$ and $c-\alg{Re}(b) \in \alg{I}_0$, which implies that $c+\alg{I}_0=\alg{Re}(b) + \alg{I}_0 \in \F^{\tau_{\A_\N}}$.

Now, consider $\psi_N(\alg{Re}(b^{N-1}_{N}))$ and recall that $\alg{Re}(b^{N-1}_{N})=\alg{Re}(b^{N}_{N}) $ and that $\alg{Re}(b^n_{n+1})=\alg{Re}(b^{n+1}_{n+1})$ for all $n \in \N$ by Proposition \ref{p:b-c*-conv-lim}.  Therefore 
\[\left\|\psi_N(\alg{Re}(b^{N-1}_{N}))-c\right\|_\D=\sup_{k \in \N} \left\|\alg{Re}(b^{N}_{N})-\alg{Re}(b^{N+k+1}_{N+k+1})\right\|_\A.\]
 Since $S_0(\alg{Re}(b))\leq r< \infty$, we have that $\|\alg{Re}(b_n^n)- \alg{Re}(b^{n+1}_{n+1})\|_\A \leq r \cdot 2 \beta(n)$ for all $n \in \N$ by Proposition \ref{p:b-c*-conv-lim}. Hence for all $k \in \N$
\begin{equation*}
\left\|\alg{Re}(b^{N}_{N})-\alg{Re}(b^{N+k+1}_{N+k+1})\right\|_\A \leq 2 r \cdot \sum_{j=N}^{N+k} \beta(j) \leq  2 r \cdot \sum_{j=N}^\infty \beta(j) < \varepsilon/4.
\end{equation*}
Thus $\|\psi_N(\alg{Re}(b^{N-1}_{N}))-c\|_\D \leq \varepsilon/4$.  Similarly, we may find $d \in \alg{L}_0$ with $\alg{Im}(b)+\alg{I}_0 =d+\alg{I}_0 $  such that $\|\psi_N(\alg{Im}(b^{N-1}_{N}))-d\|_\D \leq \varepsilon/4$. Note that $c+id \in \alg{S}_0$ with $b+\alg{I}_0= (c+id)+\alg{I}_0$ and that $\|\psi_N(b^{N-1}_N)- (c+id) \|_\D \leq \varepsilon/2$. Therefore, since $\psi(b^{N-1}_N) =\psi^{(N)}(b^{N-1}_N) )$ as $b^{N-1}_N \in \A_N$, we gather 
\begin{equation*}
\begin{split}
\|a+\alg{I}_0 - \psi(b^{N-1}_N) \|_{\F^{\tau_{\A_\N}}} & \leq \|a+\alg{I}_0 - b+\alg{I}_0\|_{\F^{\tau_{\A_\N}}}+ \| \psi(b^{N-1}_N) - b+\alg{I}_0\|_{\F^{\tau_{\A_\N}}} \\
& 
< \varepsilon/2 + \| \psi(b^{N-1}_N) - (c+id)+\alg{I}_0\|_{\F^{\tau_{\A_\N}}}\\
& \leq \varepsilon/2+\|\psi_N(b^{N-1}_N)- (c+id) \|_\D \\
& \leq \varepsilon/2+\varepsilon/2 =\varepsilon
\end{split}
\end{equation*}
by definition of quotient norm. In particular, the set $\psi(\A)$ is dense in ${\F^{\tau_{\A_\N}}}$.  As $\psi$ is an isometry and $\A$ is complete, it must be the case that $\psi(\A)={\F^{\tau_{\A_\N}}}$. 

Now, assume that $a \in  \cup_{n \in \N} \dom{\Lip_{\A_n}}$, then there exists $N \in \N, N>1$ such that $a \in \sa{\A_N}$ and  $\Lip_{\A_N}(a)< \infty$. Thus by Proposition \ref{p:b-c*-conv-lim}, 
\begin{equation*}
\begin{split}
\Lip_\A(a)&=\Lip_{\F^{\tau_{\A_\N}}}\circ \psi(a)= \Lip_{\F^{\tau_{\A_\N}}}\circ \psi^{(N)}(a) \leq S_0(\psi_N(a))\\
&= \max \{\Lip_{\A_N}(a), \|a\|_\A/(2\beta(N-1))\}< \infty.
\end{split}
\end{equation*}
The remaining follows from Proposition \ref{p:prop-ind-lim} and Proposition \ref{p:c*-cauchy}.
\end{proof}
Next, we see how our approximable inductive limits are well-suited for providing convergent sequences of inductive limits by reducing the problem of showing convergence to the terms of the inductive sequence, which will be applied in Theorem \ref{t:ideal-conv}.
\begin{theorem}\label{t:sequence}
Fix $C \geq 1, D \geq 0$.  For each $k \in \N \cup \{\infty\}$,   let $\A^k=\overline{\cup_{n \in \N} \A^k_n }^{\|\cdot \|_\A}$ be a unital   inductive limit of C*-algebras such that  $\A^k$ is $\left(\left(\A^k_n, \Lip_{\A^k_n}\right)\right)_{n \in \N}$-propinquity approximable for some sequence of $(C,D)$-quasi-Leibniz compact quantum metric spaces and summable $(\beta^k(j))_{j \in \N} \subset (0, \infty)$, and let $\left(\A^k, \Lip_{\A^k}\right)$ be the limit of $\left(\left(\A^k_n, \Lip_{\A^k_n}\right)\right)_{n \in \N}$ in dual propinquity given by (3) of Theorem \ref{t:main}.

If:
\begin{enumerate}
\item there exists a summable $(\beta(j))_{j \in \N}\subset (0, \infty)$ such that $\beta^k(j) \leq \beta(j)$ for all $k \in \N \cup \{0\},j \in \N$, and
\item for every $n \in \N$, it holds thats 
\[
\lim_{k \to \infty} \dpropinquity{} \left( \left(\A^k_n,\Lip_{\A^k_n}\right), \left(\A_n^\infty, \Lip_{\A_n^\infty}\right)\right)=0,
\]
\end{enumerate}
then
\[
\lim_{k \to \infty} \dpropinquity{} \left( \left(\A^k,\Lip_{\A^k}\right), \left(\A^\infty, \Lip_{\A^\infty}\right)\right)=0.
\]
\end{theorem}
\begin{proof}
Let $\varepsilon>0$. There exists $N \in \N$ such that $4\sum_{j=N}^\infty \beta(j) < \varepsilon/3.$  Thus by Theorem \ref{t:main} and the triangle inequality, for each $k \in \N$,
\begin{equation*}
\begin{split}
& \dpropinquity{} \left( \left(\A^k,\Lip_{\A^k}\right), \left(\A^\infty, \Lip_{\A^\infty}\right)\right)\\
& \leq \dpropinquity{} \left( \left(\A^k,\Lip_{\A^k}\right), \left(\A_N^k, \Lip_{\A_N^k}\right)\right)+\dpropinquity{} \left(  \left(\A_N^k, \Lip_{\A_N^k}\right),  \left(\A_N^\infty, \Lip_{\A_N^\infty}\right)\right)\\
& \quad +\dpropinquity{} \left(\left(\A_N^\infty, \Lip_{\A_N^\infty}\right), \left(\A^\infty, \Lip_{\A^\infty}\right)\right)\\
& \leq 4\sum_{j=N}^\infty \beta^k(j)+\dpropinquity{} \left(  \left(\A_N^k, \Lip_{\A_N^k}\right),  \left(\A_N^\infty, \Lip_{\A_N^\infty}\right)\right) +4\sum_{j=N}^\infty \beta^\infty(j)\\
& < (2\varepsilon/3) + \dpropinquity{} \left(  \left(\A_N^k, \Lip_{\A_N^k}\right),  \left(\A_N^\infty, \Lip_{\A_N^\infty}\right)\right). 
\end{split}
\end{equation*}
Therefore
\[
\limsup_{k \to \infty} \dpropinquity{} \left( \left(\A^k,\Lip_{\A^k}\right), \left(\A^\infty, \Lip_{\A^\infty}\right)\right) \leq (2\varepsilon/3)< \varepsilon,
\]
which completes the proof as $\varepsilon>0$ was arbitrary.
\end{proof}
In this process of building a Lip-norm on the inductive limit, we also created new Lip-norms on the terms on the inductive sequence itself.  We close this section with some comparisons in the next proposition that also serves to explain further the structure of these new Lip-norms.  We also consider the case when the inductive limit already comes equipped with a certain kind of Lip-norm.
\begin{proposition}\label{p:lip-compare}
Fix $C \geq 1, D \geq 0$.  Let $\A=\overline{\cup_{n \in \N} \A_n }^{\|\cdot \|_\A}$ be a unital   inductive limit of C*-algebras. Let $\A$ be  $\left(\left(\A_n, \Lip_{\A_n}\right)\right)_{n \in \N}$-propinquity approximable for some sequence of $(C,D)$-quasi-Leibniz compact quantum metric spaces and summable $(\beta(j))_{j \in \N} \subset (0, \infty)$, and let $\beta(-1):=\infty$ with the convention that $(1/\infty)=0$.

If $\left(\F^{\tau_{\A_\N}}, \Lip_{\tau_{\A_\N}}\right)$, $\psi$, and  $(\A, \Lip_\A)$ are as in Theorem \ref{t:main}, then
\begin{enumerate}
\item for each $n \in \N$, the pair $(\A_n, \Lip_{\tau_{\A_\N}} \circ \psi^{(n)})$ is a $(C,D)$-quasi-Leibniz compact quantum metric space such that
\[
\Lip_\A(a)=\Lip_{\tau_{\A_\N}} \circ \psi^{(n)}(a) \leq \max\{1, \diam{\A_n}{\Lip_{\A_n}}/\beta(n-1)\}
\cdot \Lip_{\A_n}(a)
\]
 $\text{ for all } a \in \sa{\A_n},$ and
\item  if there exists a $(C,D)$-quasi-Leibniz Lip-norm $\Lip_\A'$ on $\A$ defined and lower semi-continuous on all of $\A$ such that $\Lip_\A'=\Lip_{\A_n}$ on $\A_n$ for each $n \in \N$, then
\[
 \Lip_{\A_n}(a) \leq \Lip_{\tau_{\A_\N}} \circ \psi^{(n)}(a) \leq \max\{1, \diam{\A_n}{\Lip_{\A_n}}/\beta(n-1)\} \cdot  \Lip_{\A_n}(a)
\]
$ \text{ for all } a \in \sa{\A_n}$, equivalently
\[
 \Lip_\A'(a) \leq \Lip_\A(a) \leq \max\{1, \diam{\A_n}{\Lip_{\A_n}}/\beta(n-1)\} \cdot \Lip_\A'(a) 
\]
$\text{ for all } a \in \sa{\A_n}.$
\end{enumerate}
\end{proposition}
\begin{proof}
Fix $n \in \N$.  Let $a \in \dom{\Lip_{\A_n}}$. If $n=0$, then the conclusions are clear since the only Lip-norm on $\C$ is the 0-seminorm.  So, assume that $n \geq 1$. Then by Proposition \ref{p:b-c*-conv-lim}, we have that
\begin{equation*}
\begin{split}
\Lip_{\tau_{\A_\N}} \circ \psi^{(n)}(a) \leq S_0(\psi_n(a))= \max\{\Lip_{\A_n}(a), \|a\|_\A/\beta(n-1)\} < \infty.
\end{split}
\end{equation*}
Therefore, the domain of $\Lip_{\tau_{\A_\N}} \circ \psi^{(n)}$ is dense in $\A_n$. Now, fix a state $\mu \in \StateSpace(\F^{\tau_{\A_\N}})$.  We have that $\{d \in \sa{\F^{\tau_{\A_\N}}} \mid \Lip_{\tau_{\A_\N}}(d) \leq1 \text{ and } \mu(d)=0\}$ is compact by \cite[Proposition 1.3]{Ozawa05}.  Thus the set $\{d \in \sa{\psi^{(n)}(\A_n)} \mid \Lip_{\tau_{\A_\N}}(d) \leq1 \text{ and } \mu(d)=0\}$ is compact.  Therefore, as $\Lip_{\tau_{\A_\N}} \circ \psi^{(n)}$ is lower semi-continuous on $\A_n$ and vanishes only on scalars as $\psi^{(n)}$ is a unital *-monomorphism by Lemma \ref{l:u*-mon}, we have that $(\A_n, \Lip_{\tau_{\A_\N}} \circ \psi^{(n)})$ is a $(C,D)$-quasi-Leibniz compact quantum metric space by \cite[Proposition 1.3]{Ozawa05}  since $\psi^{(n)}$ is a *-homomorphism.

Next, fix $a \in \dom{\Lip_{\A_n}}$. As Lip-norms vanish on scalars, we have for any $\lambda \in \R$ by Proposition \ref{p:b-c*-conv-lim} and since $\psi^{(n)}$ is unital,
\begin{equation*}
\begin{split}
\Lip_{\tau_{\A_\N}} \circ \psi^{(n)}(a)&=\Lip_{\tau_{\A_\N}} \circ \psi^{(n)}(a-\lambda1_\A)  \leq S_0(\psi_n(a-\lambda1_\A))\\
& = \max \{\Lip_{\A_n}(a-\lambda1_\A), \|a-\lambda1_\A\|_\A/\beta(n-1)\} \\
& =\max \{\Lip_{\A_n}(a), \|a-\lambda1_\A\|_\A/\beta(n-1)\}.
\end{split}
\end{equation*}
Thus, as $\lambda \in \R$ was arbitrary, we have 
\begin{equation*}
\begin{split} \Lip_{\tau_{\A_\N}} \circ \psi^{(n)}(a) &  \leq \max \{\Lip_{\A_n}(a), \inf_{\lambda \in \R}\{\|a-\lambda1_\A\|_\A\}/\beta(n-1)\}\\
& \leq \max \{\Lip_{\A_n}(a), (\diam{\A_n}{\Lip_{\A_n}})/\beta(n-1)) \Lip_{\A_n}(a)\}\\
& =  \max \{1, (\diam{\A_n}{\Lip_{\A_n}})/\beta(n-1)) \}\Lip_{\A_n}(a)
\end{split}
\end{equation*}
by \cite[1.6 Proposition]{Rieffel98a}.  This completes (1).

For (2), fix $a \in \A_n$.  Next, let $d=((d_k^k, d_{k+1}^k))_{k \in \N} \in \alg{I}_0$.  Thus, the sequence $(d^k_k)_{k \in \N} \subset \cup_{l \in \N}\A_l$ converges to $0$. And, the sequence $(a -d^{k+n}_{k+n})_{k \in \N} \subset \cup_{l \in \N}\A_l$ converges to $a$.  Therefore, since $\Lip_\A'$ is lower semi-continuous on $\A$, we have
\begin{equation*}
\begin{split}
\Lip_\A'(a) \leq \liminf_{k \to \infty} \Lip_\A'(a-d^{k+n}_{k+n})= \liminf_{k \to \infty} \Lip_{\A_{k+n}}(a-d^{k+n}_{k+n}) \leq S_0(\psi_n(a)-d)
\end{split}
\end{equation*}
by Proposition \ref{p:b-c*-conv-lim}.  Thus $\Lip_\A'(a) \leq \Lip_{\tau_{\A_\N}} \circ \psi^{(n)}(a)$ as $ d\in \alg{I}_0$ was arbitrary.
\end{proof}

\section{All unital AF algebras are propinquity approximable}\label{s:af}

In this section, we use our results in the previous section to show that any unital AF algebra $\A$ is propinquity approximable with regard to any non-decreasing sequence of unital finite-dimensional C*-subaglebras $(\A_n)_{n \in \N}$ such that, we have  $\A=\overline{\cup_{n \in \N} \A_n}^{\|\cdot \|_\A}$.  This was already shown to be the case in \cite{Aguilar-Latremoliere15} only for unital AF algebras equipped with faithful tracial states. This allowed for investigation of UHF algebras \cite{Glimm60} and Effros-Shen algebras \cite{Effros80b}, but left out critical examples of AF algebras like the unitalization of the compact operators on a separable Hilbert space \cite[Example III.2.3]{Davidson} and the Boca-Mundici AF algebra \cite{Boca08, Daniele88}. Now, all unital AF algebras have been shown to be quasi-Leibniz compact quantum metric spaces \cite{Aguilar16b, Antonescu04}, yet outside the faithful tracial state case, these quantum metrics do not provide approximability. And, approximability is desirable since it allows one to prove theorems about convergence of sequences of  unital AF algebras. We display this directly in this section by showing that convergence of ideals of unital AF algebras in the Fell topology provides convergence of their unitizations in the dual propinquity topology in Theorem \ref{t:ideal-conv} with the new quasi-Leibniz quantum metrics on all unital AF algebras introduced in this section in Theorem \ref{t:main-af}.  To build these and quantum metrics, we simply use the techniques of the previous section and \cite{Aguilar-Latremoliere15} along with the observation that all finite-dimensional C*-algebras have faithful tracial states, and we find that our construction generalizes the Lip-norms of \cite{Aguilar-Latremoliere15} in Proposition \ref{p:af-lip-compare}.

\begin{convention}\label{c:conv}
We call a unital C*-algebra of the form $\A=\overline{\cup_{n \in \N} \A_n}^{\| \cdot \|_\A}$, a {\em unital AF algebra} if $(\A_n)_{n \in \N}$ is a non-decreasing  sequence of unital finite-dimensional C*-subalgebras such that $\A_0=\C1_\A$.
\end{convention}
We provide the definition of conditional expectations for convenience.
\begin{definition}[{\cite[Definition 1.5.9]{Brown-Ozawa}}]\label{td:cond-def} Let $\A$ be a unital C*-algebra and let $\B \subseteq \A$ be a unital C*-subalgebra. A linear map $E : \A \rightarrow \B$ such that $E(\A)=\B$ is a {\em conditional expectation} if 
\begin{enumerate}
\item $E(b)=b$ for all $b \in \B$, 
\item $E$ is  contractive completely positive \cite[Definition 1.5.1]{Brown-Ozawa}, and
\item  $E(bxb')=bE(x)b'$ for all $x \in \A, b,b' \in \B$.
\end{enumerate}
\end{definition}
Next, we give notation for the conditional expectations for the finite-dimen\-sional C*-algebras of a given AF algebra.
\begin{notation}\label{n:cond-exp}
Let $\A=\overline{\cup_{n \in \N} \A_n}^{\| \cdot \|_\A}$ be a unital AF algebra. For each $n \in \N$,    let $E_{n+1,n}: \A_{n+1} \rightarrow \A_n$ be a condition expectation, which always exists by \cite[Lemma 1.5.11]{Brown-Ozawa}  since  finite-dimensional C*-algebras have faithful tracial states and are von Neumann algebras.  If $n>m\geq 0$, then denote
\[
E_{n,m}:=E_{m+1,m} \circ \cdots \circ E_{n,n-1}: \A_n \rightarrow \A_m,
\]
and if $n=m$, then $E_{n,m}:=\mathrm{id}_n$, the identity map on $\A_n$. 
\end{notation}

\begin{theorem}\label{t:main-af}
Let $\A=\overline{\cup_{n \in \N} \A_n}^{\| \cdot \|_\A}$ be a unital AF algebra. Denote $\mathcal{U}=(\A_n)_{n \in \N}$. Let $(\beta(j))_{j \in \N}\subset (0, \infty)$ be summable.  For each $n \in \N$,  let $E_{n+1,n}: \A_{n+1} \rightarrow \A_n$ be a contitional expectation onto $\A_n$.

If for each $n \in \N \setminus \{0\}$, we define for all $a \in \A_n$,
\[
\Lip_{\A_n, E}^\beta(a):=\max_{m \in \{0, \ldots,  n-1\}} \left\{ \frac{\|a-E_{n,m}(a)\|_\A}{\beta(m)}\right\} 
\]
using Notation \ref{n:cond-exp} and $\Lip_{\A_0, E}^\beta$ is the $0$-seminorm on $\A_0$,  then
\begin{enumerate}
\item for each $n \in \N$, the pair $(\A_n, \Lip_{\A_n, E}^\beta)$ is a $(2,0)$-quasi-Leibniz comact quantum metric space such that $\diam{\A_n}{\Lip_{\A_n, E}^\beta} \leq 2\beta(0),$
\item $\A$ is $((\A_n, \Lip_{\A_n, E}^\beta))_{n \in \N}$-approximable with respect to $(\beta(j))_{j \in \N}\subset (0, \infty)$, and 
\item the $(2, 0)$-quasi-Leibniz Lip-norm $\Lip_{\mathcal{U}, E}^\beta$ on $\A$ given by \emph{(2)} and Theorem \ref{t:main} satisfies $\sa{\cup_{n \in \N} \A_n} \subseteq \dom{\Lip_{\mathcal{U}, E}^\beta}$, where
\[
\lim_{n \to \infty} \dpropinquity{} \left(\left(\A_n, \Lip_{\A_n, E}^\beta\right), \left(\A, \Lip_{\mathcal{U}, E}^\beta\right)\right)=0
\]
and 
\[
 \dpropinquity{} \left(\left(\A_n, \Lip_{\A_n, E}^\beta\right), \left(\A, \Lip_{\mathcal{U}, E}^\beta\right)\right) \leq 4 \sum_{j=n}^\infty \beta(j) \text{ \quad for each } n \in \N,
\]
where for each $n \in \N $, it holds that  
\[\Lip_{\mathcal{U}, E}^\beta(a) \leq \max\{1,2 \beta(0)/\beta(n-1)\}  \cdot \Lip_{\A_n, E}^\beta(a)\] for all $a \in \sa{\A_n}$, where $\beta(-1):=\infty$ with the convention that $(1/\infty)=0$. 
\end{enumerate}
In particular, every unital AF algebra $\A$ is propinquity approximable, and $\A$ is propinquity approximable by any increasing sequence $(\A_n)_{n \in \N}$ of finite-dimensional unital C*-subalgebras such that $\A_0=\C1_\A$ and $\A=\overline{\cup_{n \in \N} \A_n}^{\| \cdot \|_\A},$ and if $k \in (1, \infty)$, for each $j \in \N$, then we may choose $\beta(j)\leq 1/\dim(\A_j)^k$.
\end{theorem}
\begin{proof}
Fix $n \in \N \setminus \{0\}$.  Note that $E_{n,0}(\A_n) \subseteq \A_0=\C1_\A$ and $\A_0 \subseteq \A_m$ for all $m \in \N$. Thus $\Lip_{\A_n, E}^\beta(a)=0$ if and only if $a \in \C1_\A$. Furthermore $\{ a \in \A_n \mid \Lip_{\A_n}(a) < \infty \}=\A_n$ and $\Lip_{\A_n, E}^\beta$ is lower semi-continuous.  Also $\Lip_{\A_n, E}^\beta$ is $(2,0)$-quasi-Leibniz by \cite[Lemma 3.2]{Aguilar-Latremoliere15}.

 Denote $\sigma : \lambda1_\A \in \A_0 \mapsto \lambda \in \C$ and consider $\mu_n:=\sigma \circ E_{n,0}$, which is a state on $\A_n$.  Let $ a \in \A_n$ such that $\Lip_{\A_n, E}^\beta(a) \leq 1$ and $\mu_n(a)=0$.  Thus $E_{n,0}(a)=0$, and so $\|a\|_\A=\|a-E_{n,0}(a)\|_\A \leq \beta(0)$.  Hence the set $\{a \in \A_n \mid \Lip_{\A_n, E}^\beta(a) \leq 1 \text{ and } \mu_n(a)=0\}$ is totally bounded by finite-dimensionality. Therefore  $(\A_n, \Lip_{\A_n, E}^\beta)$ is a (2,0)-quasi-Leibniz compact quantum metric space by \cite[Proposition 1.3]{Ozawa05}. To estimate the diameter, we note that for any $a \in \A_n$ with $\Lip_{\A_n, E}^\beta(a) \leq 1$ and any $\mu, \nu \in \StateSpace(\A_n)$, 
 \begin{equation*}
 \begin{split}
 |\mu(a)-\nu(a)|&=|\mu(a-\mu_n(a)1_\A)-(\nu(a-\mu_n(a)1_\A)| = |(\mu-\nu)(a-\mu_n(a)1_\A)| \\
 &\leq 2 \|a-\mu_n(a)1_\A\|_\A = 2\|a-E_{n,0}(a)\| \leq 2 \beta(0).
 \end{split}
 \end{equation*}
 
 Thus (1) is proven. Now, Fix $n \in \N$.  Assume $a \in \A_n$.  Then $E_{n+1,n}(a)=a$ and so $E_{n+1,m}(a)=E_{n,m}(a)$ for all $m \in \{0, \ldots, n\}$. Therefore
\begin{equation}\label{e:contractive}
  \Lip_{\A_{n+1}, E}^\beta(a)= \Lip_{\A_n, E}^\beta(a) \text{ \ for all \ } a \in \A_n. 
\end{equation}

 Next, we estimate lengths of the evident bridges. Fix $n \in \N$.  Let $a \in \A_n$ such that $\Lip_{\A_n, E}^\beta(a) \leq 1$. Then $ \Lip_{\A_{n+1}, E}^\beta(a)= \Lip_{\A_n, E}^\beta(a) \leq 1$ by Equation (\ref{e:contractive}) and $\|a-a\|_\A=0$. Now, let $a \in \A_{n+1}$ such that $\Lip_{\A_{n+1}, E}^\beta(a) \leq 1$.  
 Fix $m \in \{0, \ldots, n-1\}$. Note that $E_{n+1,n}(a) \in \A_n$ and  consider 
 \begin{equation*}
 \begin{split}
 \|E_{n+1,n}(a)-E_{n,m}(E_{n+1,n}(a))\|_\A&=  \|E_{n+1,n}(a)-E_{n+1,m}(a)\|_\A\\
 & = \|E_{n+1,n}(a)-E_{n+1,n}(E_{n+1,m}(a))\|_\A \\
 &  =\|E_{n+1,n}(a-E_{n+1,m}(a))\|_\A  \leq \|a-E_{n+1,m}(a)\|_\A,
 \end{split}
 \end{equation*}
 where we used the fact that $\A_m \subseteq \A_n $ in line 2.  Thus 
 \[\Lip_{\A_n, E}^\beta(E_{n+1,n}(a)) \leq \Lip_{\A_{n+1}, E}^\beta(a) \leq 1.\]
  Furthermore, the assumption that $\Lip_{\A_{n+1}, E}^\beta(a) \leq 1$ implies that $\|a-E_{n+1,n}(a) \|_\A \leq \beta(n)$. Finally, note that if $ a\in \sa{\A_{n+1}}$, then $ E_{n+1,n}(a) \in \sa{\A_n}$ by positivity of conditional expectations. Therefore by Lemma \ref{l:bridge}, we have that the length of the evident bridge from Definition \ref{d:approx-ind} satisfies  
  \[\lambda(\gamma_{n,n+1}|\Lip_{\A_{n}, E}^\beta, \Lip_{\A_{n+1}, E}^\beta) \leq \beta(n).\]  
 
 Thus (2) is proven, and (3) follows immediately from Theorem \ref{t:main} and Proposition \ref{p:lip-compare}.  The remaining statement follows from the fact that there exist conditional expectations on any finite-dimensional C*-algebra onto any C*-subalgebra by \cite[Lemma 1.5.11]{Brown-Ozawa}, and the fact that for any increasing sequence $(\A_n)_{n \in \N}$  such that $\A_0=\C1_\A$, it holds that $\dim(\A_j) \geq j$ for all $j \in \N.$
\end{proof}
We note that for certain AF algebras, the beta sequence can be given as $\beta(n)=\frac{1}{\dim(\A_n)}$.  For example, this is the case of UHF algebras, Effros-Shen algebras, and the unitalization of the compact operators as long as the increasing sequence $\A_n$ is chosen appropriately.  However, for the C*-algebra of continuous functions on the one-point compactification on the natural numbers $C(\overline{\N})$, one could take the increasing sequence to be $\A_n \cong \C^{n+1}$, in which case taking the dimension to a power greater than 1 is necessary to produce a summable sequence.

Next, we see that in the case of a unital AF algebra with a faithful tracial state, the Lip-norms of Theorem \ref{t:main-af} can recover the Lip-norms of \cite[Theorem 3.5]{Aguilar-Latremoliere15} up to quantum isometry.  Thus, we provide a generalization of the results of \cite{Aguilar-Latremoliere15} in the case that $(\beta(j))_{j \in \N}$ is summable, and as a consequence, we  achieve convergence in the quantum Gromov-Hausdorff propinquity of \Latremoliere \  \cite{Latremoliere13} in the faithful tracial state case.
\begin{proposition}\label{p:af-lip-compare}
Let $\A=\overline{\cup_{n\in \N} \A_n}^{\|\cdot\|_\A}$ be a unital AF algebra equipped with  a faithful tracial state $\mu$.  Denote $\mathcal{U}=(\A_n)_{n \in \N}$, and let $(\beta(j))_{j \in \N} \subset (0, \infty)$ be summable.  Let
\[
\Lip_{\mathcal{U}, \mu}^\beta(a)=\sup_{n \in \N} \frac{\|a-E_n(a)\|_\A}{\beta(n)}
\]
be the (2,0)-quasi-Leibniz Lip-norm on $\A$ of \cite[Theorem 3.5]{Aguilar-Latremoliere15}, where $E_n: \A \rightarrow \A_n$ is the unique $\mu$-preserving conditional expectation onto $\A_n$.

If for each $n \in \N$, we define $E_{n+1,n}:=E_n|_{\A_{n+1}}: \A_{n+1} \rightarrow \A_n$ and let $\Lip_{\A_n,E}^\beta, \Lip_{\mathcal{U},E}^\beta$ be the associated (2,0)-quasi-Leibniz Lip-norms on $\A_n, \A$, respectively, from Theorem \ref{t:main-af}, then:
\begin{enumerate}
\item for each $n \in \N$, we have $\Lip_{\mathcal{U}, \mu}^\beta(a)=\Lip_{\A_n,E}^\beta(a)$ for all $a\in \A_n$, 
\item for each $n \in \N$ and all $a \in \sa{\A_n}$, it holds that
\[
\Lip_{\mathcal{U},\mu}^\beta(a) \leq \Lip_{\mathcal{U}, E}^\beta(a) \leq \max\{1, 2\beta(0)/\beta(n-1)\} \Lip_{\mathcal{U},\mu}^\beta(a), 
\]
where $\beta(-1):=\infty$ with the convention that $(1/\infty)=0$,
\item there exists a quantum isometry $\pi : \A \rightarrow \A$ and so
\[
\dpropinquity{}\left(\left(\A, \Lip_{\mathcal{U}, \mu}^\beta\right),\left(\A, \Lip_{\mathcal{U}, E}^\beta\right)\right)=0,
\]
\item and for each $n \in \N$, 
\[
\dpropinquity{}\left(\left(\A_n, \Lip_{\A_n, E}^\beta\right),\left(\A, \Lip_{\mathcal{U}, E}^\beta\right)\right) \leq 2 \qpropinquity{}\left(\left(\A_n, \Lip_{\A_n, E}^\beta\right),\left(\A, \Lip_{\mathcal{U}, E}^\beta\right)\right) \leq 2\beta(n),
\]
and thus $\lim_{n \to \infty } \qpropinquity{}\left(\left(\A_n, \Lip_{\A_n, E}^\beta\right),\left(\A, \Lip_{\mathcal{U}, E}^\beta\right)\right)=0$, where $\qpropinquity{}$ is the quantum Gromov-Hausdorff propinquity of \cite{Latremoliere13}.
\end{enumerate} 
\end{proposition}
\begin{proof}
Fix $n \in \N \setminus \{0\}$. By the proof of \cite[Theorem 3.5]{Aguilar-Latremoliere15}, it holds tht $E_n \circ E_{n-1}=E_{n-1}$.  Therefore $E_{n,m}=E_m$ for all $m \in \{0, \ldots, n-1\}$.  Thus (1) is proven. Conclusions (2) and (3) follow from Theorem \ref{t:main-af} and \cite[Theorem 3.5]{Aguilar-Latremoliere15} and the fact that limits are unique up to quantum isometry in the dual propinquity. Conclusion (4) follows from \cite[Theorem 5.5]{Latremoliere13b} and \cite[Theorem 3.5]{Aguilar-Latremoliere15}.
\end{proof}

We can now apply our findings to discover new continuous families of unital AF algebras since we now do not require faithful tracial states.  We will form a continuous map from the ideal space of any unital AF algebra into the dual propinquity space by mapping an ideal to its unitization with carefully chosen Lip-norms  from Theorem \ref{t:main-af}. So, although we cannot discuss convergence of ideals themselves since ideals may not contain a unit, we can now in the very least capture the Fell topology by considering the unitization of the ideals.  This was not possible before due to the fact that  not every unital AF algebra has a faithful tracial state.  However, since an ideal may or may not contain unit, we have to be careful by what we mean by the "unitization" as these seem to differ when considering a unital versus a non-unital C*-algebra.  Yet, thankfully, there does exist a unitization process of C*-algebras that does not distinguish between the unital and non-unital case, which allows us to proceed much more smoothly in our construction of Lip-norms.  We present this unitization now.  The only known reference that we know of this is from \url{https://mathoverflow.net/questions/210025/unitization-process-of-unital-and-non-unital-c-algebras} due to user UwF. Thus, we provide details for the proof found in the above link. 
\begin{theorem-definition}\label{td:unit}
Let $\A$ be a C*-algebra with or without unit. The {\em unitization} of $\A$, denoted $\widetilde{\A},$ is the vector space $\A \oplus \C$ equipped with
\begin{enumerate}
\item adjoint defined by $(a,\lambda)^*:=(a^*, \overline{\lambda})$ for all $a \in \A, \lambda \in \C$, 
\item multiplication defined by $(a, \lambda)(b, \mu):=(ab+\mu a+\lambda b,\lambda \mu)$  for all $a,b \in \A, \lambda \in \C$, 
\item and norm defined by \[\|(a,\lambda)\|_{\widetilde{\A}}:=\max\left\{\sup_{b \in \A, \|b\|_\A \leq 1} \left\{\|ab+\lambda b\|_\A\right\}, \ \ |\lambda| \right\},\]
for all $a \in \A, \lambda \in \C$,
\end{enumerate}
where $\widetilde{\A}$ equipped with these operations and norm is a unital C*-algebra with unit $1_{\widetilde{\A}}=(0,1)$ and $\|a\|_\A=\|(a,0)\|_{\widetilde{\A}}$ for all $a \in \A$ and $\A$ is *-isomorphic to the maximal ideal of co-dimension one of $\widetilde{\A}$ given by $\{(a, \lambda) \in \widetilde{\A} \mid a \in \A, \lambda=0\}$.

Furthermore:
\begin{enumerate}
\item[(4)] if $\B$ is a C*-subalgebra of $\A$, then $\widetilde{\B}$ is the unital  C*-subalgebra of $\widetilde{\A}$ given by $\{(b,\lambda) \in \widetilde{\A} \mid b \in \B, \lambda \in \C\}$, 
\item[(5)] if $\A$ is non-unital, then $\|(a,\lambda)\|_{\widetilde{\A}}=\sup_{b \in \A, \|b\|_\A \leq 1} \left\{\|ab+\lambda b\|_\A\right\}$ for all $a \in \A, \lambda \in \C$, and
\item[(6)] if $\A$ is unital, then $\|(a,\lambda)\|_{\widetilde{\A}}=\max \left\{ \|a+\lambda 1_\A\|_\A, |\lambda|\right\}$  for all $a \in \A, \lambda \in \C$, in which case the direct sum C*-algebra  $\A \oplus \C$ given by coordinate-wise operations and max norm is *-isomorphic to $\widetilde{\A}$ by the *-isomorphism 
\[
(a, \lambda) \in \A \oplus \C \longmapsto \left(a-\lambda1_\A, \lambda \right) \in \widetilde{\A}.
\]
\end{enumerate}
\end{theorem-definition}
\begin{proof}
It is routine to check  that $\widetilde{\A}$ is a *-algebra with respect to the defined operations. Let $\B(\A\oplus \C)$ be the unital Banach algebra of bounded operators from the direct sum C*-algebra $\A \oplus \C$ (given by coordinate-wise operations and max norm) to itself. Denote the unit of  $\B(\A\oplus \C)$ by $1_{\B(\A\oplus \C)}$. For each $(a, \lambda)\in \A \oplus \C$, let $L_{(a,\lambda)} \in \B(\A\oplus \C)$ denote the left multiplication operator, that is $L_{(a,\lambda)} (b,\mu)=(ab, \lambda \mu)$ for all $(b, \mu) \in \A \oplus \C$.  By construction, the map
\[
(a, \lambda) \in \widetilde{\A} \longmapsto L_{(a,\lambda)}+\lambda 1_{\B(\A\oplus \C)} \in \B(\A\oplus \C)
\]
is a linear multiplicative function. Now, we show it is injective.  Assume that $(a, \lambda) \in \widetilde{\A}$ such that $L_{(a,\lambda)}+\lambda 1_{\B(\A\oplus \C)}=0$.  Thus $(L_{(a,\lambda)}+\lambda 1_{\B(\A\oplus \C)})(0,1)= (0, \lambda) \implies (0,0)=(0,\lambda) \implies \lambda=0$.  Hence $(0,0)=(L_{(a,\lambda)}+\lambda 1_{\B(\A\oplus \C)})(a^*,0)=L_{(a,0)}(a^*,0)=(aa^*,0) \implies aa^*=0 \implies a=0$. For all $(a, \lambda) \in \widetilde{\A}$, we gather that
\begin{equation*}
\begin{split}
&\left\|L_{(a,\lambda)}+\lambda 1_{\B(\A\oplus \C)}\right\|_{\B(\A\oplus \C)}\\
&  =\sup_{(b,\mu) \in  \A\oplus \C, \|(b,\mu) \|_{\A \oplus \C} \leq 1} \left\{ \left\|\left(L_{(a,\lambda)}+\lambda 1_{\B(\A\oplus \C)}\right)(b, \mu) \right\|_{\A \oplus \C}\right\}\\
& =\sup_{b \in  \A, \mu \in  \C, \|b\|_\A \leq 1, |\mu|\leq 1} \left\{ \left\|\left(ab+\lambda b, \lambda \mu\right) \right\|_{\A \oplus \C}\right\}\\
& = \sup_{b \in  \A, \mu \in  \C, \|b\|_\A \leq 1, |\mu|\leq 1} \left\{ \max \left\{\|ab+\lambda b\|_\A, |\lambda \mu|\right\}\right\}\\
& = \sup_{b \in  \A, \|b\|_\A \leq 1} \left\{ \max \left\{\|ab+\lambda b\|_\A, |\lambda |\right\}\right\}=\|(a,\lambda)\|_{\widetilde{\A}}.
\end{split}
\end{equation*}
Therefore $\|\cdot \|_{\widetilde{\A}}$ defines a Banach algebra norm on $\widetilde{\A}$. The fact that this norm satisfies the C*-identity follows the same argument as the proof of the standard unitization \cite[Proposition I.1.3]{Davidson}.

For (4), cleary the set $\{(b,\lambda) \in \widetilde{\A} \mid b \in \B, \lambda \in \C\}$ is a unital C*-subalgebra of $\widetilde{\A}$ that equals $\widetilde{\B}$. Since they have the same algebraic operations and adjoint, their norms also agree since C*-norms are unique.  

Conclusion (5) follows similiary since the expression on the right-hand side defines a C*-norm on $\widetilde{\A}$, when $\A$ is non-unital by \cite[Proposition I.1.3]{Davidson}.

Finally, for (6), when $\A$ is unital, we have that for $b \in \A, \|b\|_\A \leq 1$
\[
\|ab+\lambda b\|_\A =\|(a+\lambda1_\A)b\|_\A \leq \|a+\lambda 1_\A\|_\A,
\]
and since $1_\A \in \A, \|1_\A\|_\A=1$, we have the desired equality for the norm.  It is routine to show that the map in (6) is a *-isomorphism.
\end{proof}

Although we do not define the Fell topology on the ideals of a C*-algebra, we provide some results about the topology. For a more detailed summary see \cite[Section 3]{Aguilar16}.
\begin{theorem-definition}[{\cite[Theorem 2.2]{Fell61}}]\label{td:Fell} Let $\A$ be a C*-algebra. Let $\mathrm{Ideal}(\A)$ denote the set of norm-closed two-sided ideals of $\A$ including $\{0\}, \A$.  The {\em Fell topology on $\mathrm{Ideal}(\A)$} is a compact Hausdorff topology such that a net $(I_\mu)_{\mu \in \Delta} \subseteq \mathrm{Ideal}(\A)$ converges to $I \in \mathrm{Ideal}(\A)$ if and only if for each $a \in \A$ the net $(\|a+I_\mu\|_{\A/I_\mu})_{\mu \in \Delta} \subseteq \R$ converges to $\|a+I\|_{\A/I} \in \R$ with respect to the usual topology on $\R$.
\end{theorem-definition}
On primitive ideals, the Fell topology is stronger than the Jacobson topology \cite[Proposition 3.7]{Aguilar16}, and is in fact built from the closed subsets of the primitive ideals in the Jacobson topology. The Fell topology is a general construction of a topology on closed subsets of a given topology introduced by Fell in  \cite{Fell62}.  The next summarizes a topology introduced in \cite{Aguilar16} on inductive limits of C*-algebras built as an inverse limit topology from the Fellt topologies on the ideals of the C*-algebras of the inductive seqeunce, and this topology happens to agree with the Fell  topology in the case of AF algebras and is in general stronger than the Fell topology.
\begin{proposition}[{\cite[Proposition 3.15 and Theorem 3.22]{Aguilar16}}]\label{p:fell-af}
Let $\A=\overline{\cup_{n \in \N} \A_n}^{\| \cdot \|_\A}$ be an inductive limit of C*-algebras (not necessarily unital). Denote $\mathcal{U}=(\A_n)_{n \in \N}$.  The following hold.
\begin{enumerate}
\item For each $n \in \N$, the map $\iota_{n+1_i}: J \in \mathrm{Ideal}(\A_{n+1}) \mapsto J \cap \A_n \in \mathrm{Ideal}(\A_n)$ is continuous with respect to the associated Fell topologies on $\mathrm{Ideal}(\A_{n+1}) $ and $\mathrm{Ideal}(\A_{n}) $.
\item The inverse limit $\mathrm{Ideal}(\A)_\infty \subseteq \prod_{n \in \N} \mathrm{Ideal}(\A_n)$ with its topology  given by the inverse limit sequence $(\mathrm{Ideal}(\A_n), Fell, \iota_{n+1_i})_{n \in \N}$ (\cite[29.9 Definition]{Willard}) provides a topology on $\mathrm{Ideal}(\A)$ denoted $Fell_{i(\mathcal{U})}$ by the initial topology of the injection
\[
J \in \mathrm{Ideal}(\A) \mapsto (J \cap \A_n)_{n \in \N} \in  \mathrm{Ideal}(\A)_\infty.
\]
\item The topology $Fell_{i(\mathcal{U})}$ is stronger than $Fell$ on $\mathrm{Ideal}(\A)$, and if $\A$ is AF, then $Fell_{i(\mathcal{U})}=Fell$ and is metrized by the metric defined for each $I,J \in \mathrm{Ideal}(\A)$ by
\[
\mathsf{m}_{i(\mathcal{U})}(I,J)=\begin{cases}
0 & : I \cap \A_n=J\cap \A_n \text{ for all } n \in \N\\
2^{-\min \{n \in \N: I \cap \A_n \neq J \cap \A_n\}} & : \text{otherwise}.  
\end{cases}
\]
\end{enumerate}
\end{proposition}
Now, we are ready to prove our result about ideal convergence. 
\begin{theorem}\label{t:ideal-conv}
Let $\A=\overline{\cup_{n \in \N} \A_n}^{\|\cdot \|_\A}$ be a unital AF algebra. Denote $\mathcal{U}=(\A_n)_{n \in \N}$. Using  Theorem-Definition \ref{td:unit}, for each $n \in \N \setminus \{0\}$ and each $J \in \mathrm{Ideal}(\A_{n+1})$, fix a conditional expectation $E^J_{n+1,n}: \widetilde{J} \rightarrow \widetilde{J \cap \A_{n}}$,  and for each $J \in \mathrm{Ideal}(\A_1)$, fix a conditional expectation $E^J_{1,0} : \widetilde{I} \rightarrow \C1_{\widetilde{I}}.$ Fix a summable sequence $(\beta(j))_{j \in \N} \subset (0, \infty)$.

Next, for every $I\in \mathrm{Ideal}(\A)$, set $\mathcal{U}_{\widetilde{I}}=( \widetilde{I}_n)_{n \in \N},$ where for $n \in \N \setminus \{0\}$, we set $\widetilde{I}_n=\widetilde{I \cap \A_n}$ and $\widetilde{I}_0=\C1_{\widetilde{I}}$,  and note that $\widetilde{I}=\overline{\cup_{n \in \N} \widetilde{I}_n}^{\| \cdot \|_{\widetilde{I}}}$, where $\mathcal{U}_{\widetilde{I}}$ is an non-decreasing sequence of finite-dimensional unital C*subalgebras of $\widetilde{I}$ by Theorem-Definition \ref{td:unit}.   For each, $n \in \N \setminus \{0\}$, define for all $(a, \lambda) \in \widetilde{I}_n$, 
\[
\Lip^{\beta}_{ \widetilde{I}_n,E^{ \widetilde{I}}}(a, \lambda):=\max_{m \in \{0, \ldots,n-1\}} \left\{\frac{\|(a, \lambda)- E^{I \cap \A_{m+1}}_{m+1,m}\circ \cdots \circ E^{I\cap \A_n}_{n,n-1}(a, \lambda)\|_{\widetilde{I}}}{\beta(m)}\right\},
\]
and let $\Lip^{\beta}_{ \widetilde{I}_0,E^{ \widetilde{I}}}$ be the zero seminorm on $\widetilde{I}_0$.

If  for every $I\in \mathrm{Ideal}(\A)$, we define $\Lip^{\beta}_{\mathcal{U}_{\widetilde{I}},E^{ \widetilde{I}}}$ as in Theorem \ref{t:main-af},
then, the map 
\[
\mathcal{F}_\A: I \in \mathrm{Ideal}(\A) \longmapsto \left(\widetilde{I}, \Lip^{\beta}_{\mathcal{U}_{\widetilde{I}},E^{ \widetilde{I}}}\right) \in \mathrm{qLcqms}_{(2,0)}
\]
is continuous, where $\mathrm{Ideal}(\A)$ is equipped with the Fell topology and $\mathrm{qLcqms}_{(2,0)}$ is the class of $(2,0)$-quasi-Leibniz compact quantum metric spaces equipped with the topology induced by the dual propinquity $\dpropinquity{}$.  Furthermore, if $\beta(j)\leq 2^{-j-5}$ for all $j \in \N$, then the map $\mathcal{F}_\A$ is Lipschitz with respect to the metric on ideals in (3) of Proposition \ref{p:fell-af} and the dual propinquity.
\end{theorem}
\begin{proof}
Let $(I(n))_{n \in \N \cup \{\infty\}} \subseteq \mathrm{Ideal}(\A)$ be a sequence that converges to $I(\infty)$ with respect to the Fell topology of Theorem-Definition \ref{td:Fell}. We will use Theorem \ref{t:sequence}. Fix $M \in \N\setminus \{0\}.$ By Proposition \ref{p:fell-af}, there exists $N \in \N$ such that $\mathsf{m}_{i(\mathcal{U})}(I(n), I(\infty)) < 2^{-(M+1)}$ for all $n \geq N$.  Therefore, for all $n \geq N$, it holds that $I(n) \cap \A_k=I(\infty) \cap \A_k$ and $\widetilde{I(n)}_k=\widetilde{I(\infty)}_k$ for all $k \in \{0, \ldots, M\}.$ Therefore, by construction, we have that if $n \geq N$, then 
\begin{equation*}
\begin{split}
&\Lip^{\beta}_{ \widetilde{I(n)}_M,E^{ \widetilde{I(n)}}}(a, \lambda)\\
&= \max_{m \in \{0, \ldots,M-1\}} \left\{\frac{\|(a, \lambda)- E^{I(n) \cap \A_{m+1}}_{m+1,m}\circ \cdots \circ E^{I(n)\cap \A_M}_{M,M-1}(a, \lambda)\|_{\widetilde{I(n)}}}{\beta(m)}\right\}\\
&= \max_{m \in \{0, \ldots,M-1\}} \left\{\frac{\|(a, \lambda)- E^{I(n) \cap \A_{m+1}}_{m+1,m}\circ \cdots \circ E^{I(n)\cap \A_M}_{M,M-1}(a, \lambda)\|_{\widetilde{I(n)}_M}}{\beta(m)}\right\}\\
&= \max_{m \in \{0, \ldots,M-1\}} \left\{\frac{\|(a, \lambda)- E^{I(\infty) \cap \A_{m+1}}_{m+1,m}\circ \cdots \circ E^{I(\infty)\cap \A_M}_{M,M-1}(a, \lambda)\|_{\widetilde{I(\infty)}_M}}{\beta(m)}\right\}\\
&= \max_{m \in \{0, \ldots,M-1\}} \left\{\frac{\|(a, \lambda)- E^{I(\infty) \cap \A_{m+1}}_{m+1,m}\circ \cdots \circ E^{I(\infty)\cap \A_M}_{M,M-1}(a, \lambda)\|_{\widetilde{I(\infty)}}}{\beta(m)}\right\}\\
& =\Lip^{\beta}_{ \widetilde{I(\infty)}_M,E^{ \widetilde{I(\infty)}}}(a,\lambda)
\end{split}
\end{equation*}
 for all $(a, \lambda) \in \widetilde{I(n)}_M=\widetilde{I(\infty)}_M$ by (4) of Theorem-Definition \ref{td:unit}.  In particular, the identity map is a quantum isometry between $\left( \widetilde{I(n)}_M,\Lip^{\beta}_{ \widetilde{I(n)}_M,E^{ \widetilde{I(n)}}}\right)$ and $\left(\widetilde{I(\infty)}_M,\Lip^{\beta}_{ \widetilde{I(\infty)}_M,E^{ \widetilde{I(\infty)}}}\right)$ for all $n \geq N$.  Hence
\[
\lim_{n \to \infty} \dpropinquity{}\left(\left( \widetilde{I(n)}_M,\Lip^{\beta}_{ \widetilde{I(n)}_M,E^{ \widetilde{I(n)}}}\right), \left(\widetilde{I(\infty)}_M,\Lip^{\beta}_{ \widetilde{I(\infty)}_M,E^{ \widetilde{I(\infty)}}}\right)\right)=0.
\]
The same result is immediate for $M=0$.
Thus, as $M \in \N$ was arbitrary, we have that
\[
\lim_{n \to \infty} \dpropinquity{} \left(\left( \widetilde{I(n)},\Lip^{\beta}_{\mathcal{U}_{\widetilde{I(n)}},E^{ \widetilde{I(n)}}}\right), \left(\widetilde{I(\infty)},\Lip^{\beta}_{\mathcal{U}_{\widetilde{I(\infty)}},E^{ \widetilde{I(\infty)}}}\right)\right)=0
\]
by Theorem \ref{t:sequence}. Hence, the desired map is continuous.  

Next, assume that $\beta(j) \leq 2^{-j-5}$ for all $j \in \N$.  Fix $I,J \in \mathrm{Ideal}(\A)$ such that $I \neq J$.  Therefore, there exists $n \in \N$ such that $\mathsf{m}_{i(\mathcal{U})}(I,J)=2^{-n}$. Assume that $n \geq 1$. By the same argument as above and Theorem \ref{t:main-af}, we have
\begin{equation*}
\begin{split}
&  \dpropinquity{} \left(\left( \widetilde{I},\Lip^{\beta}_{\mathcal{U}_{\widetilde{I}},E^{ \widetilde{I}}}\right), \left(\widetilde{J},\Lip^{\beta}_{\mathcal{U}_{\widetilde{J}},E^{ \widetilde{J}}}\right)\right)\\
& \leq \dpropinquity{} \left(\left( \widetilde{I},\Lip^{\beta}_{\mathcal{U}_{\widetilde{I}},E^{ \widetilde{I}}}\right), \left(\widetilde{I}_{n-1},\Lip^{\beta}_{{\widetilde{I}_{n-1}},E^{ \widetilde{I}}}\right)\right) + \dpropinquity{} \left(\left(\widetilde{I}_{n-1},\Lip^{\beta}_{{\widetilde{I}_{n-1}},E^{ \widetilde{I}}}\right), \left(\widetilde{J}_{n-1},\Lip^{\beta}_{{\widetilde{J}_{n-1}},E^{ \widetilde{J}}}\right)\right)\\
& \quad + \dpropinquity{} \left(\left( \widetilde{J},\Lip^{\beta}_{\mathcal{U}_{\widetilde{J}},E^{ \widetilde{J}}}\right), \left(\widetilde{J}_{n-1},\Lip^{\beta}_{{\widetilde{J}_{n-1}},E^{ \widetilde{J}}}\right)\right)\\
& \leq 4\sum_{j=n-1}^\infty 2^{-j-5} +0 + 4\sum_{j=n-1}^\infty 2^{-j-5}= (1/4) (4\cdot 2^{-n})= \mathsf{m}_{i(\mathcal{U})}(I,J).
\end{split}
\end{equation*}
The case $n=0$ follows similarly. 
\end{proof}
We now make a remark about a certain choice we made in this section. Given a unital AF algebra $\A=\overline{\cup_{n \in \N} \A_n}^{\| \cdot \|_\A}$ with a non-decreasing sequence of unital finite-dimensional C*-subalgebras $(\A_n)_{n \in \N}$ of $\A$,   we could have used the conditional expectations from $\A_{n+1}$ onto  $\A_n$  of Theorem \ref{t:main-af} to build conditional expectations from $\A$ onto $\A_n$  that behave much like the ones in the faithful tracial state case of Proposition \ref{p:af-lip-compare}.  This would build Lip-norms on $\A$ that allow for propinquity approximability with respect to quantum propinquity and not just dual propinquity in the case of any unital AF algebra, which is not that much of an advantage since the dual propinquity is preferred anyway as it's complete.  However, to extend the conditional expectations in such a way to accomplish this,  one would essentailly be using the injectivity of finite-dimensional C*-algebras \cite[Proposition IV.2.1.4, II.6.9.13, and Arveson's Extension Theorem II.6.9.12]{Blackadar06}. Although the main example of this paper is unital AF algebras, the goal of this paper is to provide a more general framework for future work on providing Lip-norms on inductive limits that allow for not only convergence of inductive sequences to inductive limits but also convergence of sequences of inductive limits, and injectivity of C*-algebras is certainly not a common property especially among the standard classes of inductive limits studied. Hence, we chose to present our work in this section without using the injectivity of finite-dimensional C*-algebras to set the stage for future work.

We conclude this paper by noting that the results in \cite{Aguilar-Latremoliere15} are all true in this more general setting by Proposition \ref{p:af-lip-compare}.  One just simply needs replace the convergence conditions on the sequence $(\beta(n))_{n \in \N}$ with summability conditions when working in the non-faithful tracial state case.  For instance, the compactness results of \cite[Section 6]{Aguilar-Latremoliere15} and the convergence of families of Lip-norms on a fixed unital AF algebra result of \cite[Section 8]{Aguilar-Latremoliere15} are true wih this minor change, and we note that as a result \cite[Section 8]{Aguilar-Latremoliere15} provides continuous families of Lip-norms on the unitization of the compact operators by varying the  sequence $(\beta(n))_{n \in \N}$ in a suitable topology. 

\bibliographystyle{plain}
\bibliography{../../thesis-a}

 \vfill 
\end{document}